\documentclass[12pt,reqno]{amsart}

\usepackage{color}
\usepackage{amsmath, amsthm, amssymb}
\usepackage{amsfonts}

\usepackage{hyperref}

\usepackage[ansinew]{inputenc}
\usepackage[dvips]{epsfig}
\usepackage{graphicx}
\usepackage[english]{babel}

\theoremstyle{plain}
\newtheorem{thm}{Theorem}[section]

\newtheorem{lem}[thm]{Lemma}
\newtheorem{prop}[thm]{Proposition}

\theoremstyle{definition}
\newtheorem{defi}[thm]{Definition}

\theoremstyle{remark}
\newtheorem{rem}[thm]{Remark}

\numberwithin{equation}{section}

\newcommand{\average}{{\mathchoice {\kern1ex\vcenter{\hrule height.4pt
width 6pt depth0pt} \kern-9.7pt} {\kern1ex\vcenter{\hrule
height.4pt width 4.3pt depth0pt} \kern-7pt} {} {} }}
\newcommand{\ave}{\average\int}
\def\R{\mathbb{R}}
\newcommand{\I}{{\rm I}}

\begin{document}

\title[Obstacle problems for integro-differential operators]{\hspace{-20mm}Obstacle problems for integro-differential operators:\hspace{-20mm} \vspace{1mm}\\ Regularity of solutions and free boundaries}

\author[L. Caffarelli]{Luis Caffarelli}
\address{The University of Texas at Austin, Department of Mathematics, 2515 Speedway, Austin, TX 78751, USA}
\email{caffarel@math.utexas.edu}

\author[X. Ros-Oton]{Xavier Ros-Oton}
\address{The University of Texas at Austin, Department of Mathematics, 2515 Speedway, Austin, TX 78751, USA}
\email{ros.oton@math.utexas.edu}

\author[J. Serra]{Joaquim Serra}
\address{Universitat Polit\`ecnica de Catalunya, Departament de Matem\`{a}tiques, Avda. Diagonal 647, 08028 Barcelona, Spain}
\email{joaquim.serra@upc.edu}

\thanks{XR was supported by NSF grant DMS-1565186. XR and JS were supported by MINECO grant MTM2014-52402-C3-1-P (Spain)}



\maketitle

\begin{abstract}
We study the obstacle problem for integro-differential operators of order $2s$, with $s\in (0,1)$.

Our main result establish that the free boundary is $C^{1,\gamma}$ and $u\in C^{1,s}$ near all regular points.
Namely, we prove the following dichotomy at all free boundary points $x_0\in\partial\{u=\varphi\}$:\vspace{1mm}
\begin{itemize}
\item[(i)] either \quad$u(x)-\varphi(x)=c\,d^{1+s}(x)+o(|x-x_0|^{1+s+\alpha})$ \quad for some $c>0$,\vspace{1mm}
\item[(ii)] or \quad$u(x)-\varphi(x)=o(|x-x_0|^{1+s+\alpha})$,\vspace{1mm}
\end{itemize}
where $d$ is the distance to the contact set $\{u=\varphi\}$.
Moreover, we show that the set of free boundary points $x_0$ satisfying (i) is open, and that the free boundary is $C^{1,\gamma}$ and $u\in C^{1,s}$ near those points.

These results were only known for the fractional Laplacian \cite{CSS}, and are completely new for more general integro-differential operators.
The methods we develop here are purely nonlocal, and do not rely on any monotonicity-type formula for the operator.
Thanks to this, our techniques can be applied in the much more general context of fully nonlinear integro-differential operators: we establish similar regularity results for obstacle problems with convex operators.
\end{abstract}

\tableofcontents

\section{Introduction}

Obstacle problems for integro-differential operators appear naturally when considering optimal stopping problems for L\'evy processes with jumps.
Indeed, the value function $u(x)$ in this type of problems will solve $\min(-Lu,\,u-\varphi)=0$, where $\varphi$ is a certain payoff function and the operator $L$ is the infinitesimal generator of the process.
The equation can be posed either in a domain $\Omega\subset \R^n$ or in the whole space.
By the L\'evy-Khintchine formula, for any symmetric L\'evy process we have
\[Lu(x)=\sum_{i,j}a_{ij}\partial_{ij}u+\int_{\R^n}\left(\frac{u(x+y)+u(x-y)}{2}-u(x)\right)\nu(dy),\]
where $(a_{ij})$ is nonnegative definite, and $\nu$ satisfies $\int\min(1,|y|^2)\nu(dy)<\infty$.

An important motivation for studying this type of problems comes from mathematical finance, where they arise as pricing models for American options.
In this context, the function $u$ represents the rational price of a perpetual option, $\varphi$~is the payoff function, and the set $\{u=\varphi\}$ is the exercise region; see for example \cite{CT} for detailed description of the model.

When the matrix $(a_{ij})$ is uniformly elliptic, then the local term $a_{ij}\partial_{ij}u$ dominates.
In particular, if no jump part is present (i.e., $\nu\equiv0$) then after an affine change of variables we have $L=\Delta$, and the regularity of solutions and free boundaries is well understood.
However, when there is no diffusion part (i.e., $(a_{ij})=0$), then the problem is much less understood.

When $\nu(dy)=c|y|^{-n-2s} dy$ ---and $(a_{ij}) \equiv 0$--- then $L$ is a multiple of the fractional Laplacian $(-\Delta)^s$, and the obstacle problem was studied by Silvestre in \cite{S-obst} and by Caffarelli, Salsa and Silvestre in \cite{CSS}.
The main results of \cite{CSS,S-obst} establish that solutions $u$ are $C^{1,s}$, and that the free boundary is $C^{1,\alpha}$ at regular points (those at which $\sup_{B_r}(u-\varphi)\sim r^{1+s}$).
More recently, the singular set was studied in \cite{GP} for $s=\frac12$, and the complete structure of the free boundary was obtained in \cite{BFR} under a concavity assumption on the obstacle.

The proofs of all these results rely very strongly on certain particular properties of $(-\Delta)^s$.
Indeed, the obstacle problem for this (nonlocal) operator is equivalent to a thin obstacle problem in $\R^{n+1}$ for a \emph{local} operator, for which Almgren-type and other monotonicity formulas are available.
In \cite{S-obst}, the main results are established by using the semigroup property $(-\Delta)^{1-s}(-\Delta)^s=-\Delta$, thus getting a local operator.

For more general nonlocal operators $L$, for which these tools are not available, almost nothing was known about the regularity of solutions to obstacle problems, and nothing about the corresponding free boundaries.

The aim of this paper is to establish new regularity results for a general class of integro-differential operators of order $2s$, $s\in(0,1)$.
More precisely, we prove that
\begin{enumerate}
\item[(i)] solutions $u$ are $C^{1,s}$ near regular points, and   \vspace{1mm}
\item[(ii)] the free boundary is $C^{1,\alpha}$ near regular points,
\end{enumerate}
for all operators of the form
\begin{equation}\label{L*1}
Lu(x) = \int_{\R^n} \left(\frac{u(x+y)+u(x-y)}{2}-u(x)\right) \frac{\,\mu(y/|y|)}{|y|^{n+2s}}\,dy,
\end{equation}
with
\begin{equation} \label{L*2}
\mu\in L^\infty(S^{n-1}) \quad \mbox{satisfying} \quad \mu(\theta)=\mu(-\theta)\quad \mbox{and} \quad \lambda\le \mu\le \Lambda.
\end{equation}
We denote $\mathcal L_*$ the class of all linear elliptic operators \eqref{L*1}-\eqref{L*2}.
($\mathcal L_*$ consists of homogeneous translation invariant operators.)

\addtocontents{toc}{\protect\setcounter{tocdepth}{1}}  
\subsection{Main results}

Given $L\in \mathcal L_*$, we consider the obstacle problem in all of $\R^n$
\begin{equation}\label{obst-pb}\begin{split}
\min(-Lu,\,u-\varphi)&=0\quad\textrm{in}\ \, \R^n,\\
\lim_{|x|\rightarrow \infty}u(x)&=0.
\end{split}\end{equation}
The solution $u$ to \eqref{obst-pb} can be constructed as the smallest supersolution $u$ lying above the obstacle $\varphi$ and being nonnegative at infinity.

We assume that the obstacle satisfies
\begin{equation}\label{obstacle}
\varphi\textrm{ is bounded, }\, \varphi\in C^{2,1}(\R^n),\,\textrm{ and }\, \{\varphi>0\}\subset\subset\R^n.
\end{equation}
Our main result is the following.

\begin{thm}\label{th2}
Let $L$ be any operator of the form \eqref{L*1}-\eqref{L*2}, and $\alpha\in(0,s)$ be such that $1+s+\alpha<2$.

Let $\varphi$ be any obstacle satisfying \eqref{obstacle}, and $u$ be the solution to \eqref{obst-pb}.
Let $d(x)={\rm dist}(x,\{u=\varphi\})$.
Then, for every free boundary point $x_0\in \partial\{u=\varphi\}$ we have
\begin{itemize}
\item[(i)] either \vspace{-1mm}
\[u(x)-\varphi(x)=c\,d^{1+s}(x)+o(|x-x_0|^{1+s+\alpha}),\]
with $c>0$,\vspace{2mm}
\item[(ii)] or \hspace{30mm} $u(x)-\varphi(x)=o(|x-x_0|^{1+s+\alpha})$.
\end{itemize}
Moreover, the set of points $x_0$ satisfying (i) is an open subset of the free boundary and it is locally a $C^{1,\gamma}$ graph for all $\gamma\in(0,s)$.
Furthermore, for every $x_0$ satisfying (i) there is $r>0$ such that $u\in C^{1,s}(B_r(x_0))$.
\end{thm}

As explained above, these results were only known for the fractional Laplacian; see \cite{CSS}.
In that case, one can transform the problem into a thin obstacle problem for a local operator and use Almgren-type monotonicity formulas.
This is not possible for more general nonlocal operators, and new techniques had to be developed.

The proofs we present here are purely nonlocal and are independent from the ones in \cite{CSS}.
Moreover, we do not use any particular monotonicity-type formulas for the operators.
Thanks to this, our techniques can be applied in the much more general setting of fully nonlinear integro-differential equations, as explained next.

\subsection{Fully nonlinear equations}

We also establish similar regularity results for convex fully nonlinear operators
\begin{equation}\label{I}
\I u=\sup_{a\in\mathcal A} \bigl(L_au+c_a\bigr),
\end{equation}
with $L_a\in\mathcal L_*$ for all $a\in\mathcal A$.
For simplicity, we assume $I0=0$.

Given such an operator $\I$, we consider the obstacle problem
\begin{equation}\label{obst-pb-fully}\begin{split}
\min(-\I u,\,u-\varphi)&=0\quad\textrm{in}\ \, \R^n,\\
\lim_{|x|\rightarrow \infty}u(x)&=0,
\end{split}\end{equation}
for an obstacle $\varphi$ satisfying \eqref{obstacle}.

In the next result, and throughout the paper, we use the following.

\begin{defi}\label{defi-baralpha}
We denote $\bar\alpha=\bar\alpha(n,s,\lambda,\Lambda)>0$ the minimum of the three following constants:
\begin{itemize}
\item The $\alpha>0$ of the interior $C^{\alpha}$ estimate for nonlocal equations ``with bounded measurable coeffiecients'' given by \cite[Theorem 11.1]{CS};
\item The $\alpha>0$ of the boundary $C^{\alpha}$ estimate for $u/d^s$ for the same equations given by \cite[Proposition 1.1]{RS-K};
\item The $\alpha>0$ in the interior  $C^{2s+\alpha}$ estimate for convex equations 
given by \cite[Theorem 1.1]{CS2} and \cite[Theorem 1.1]{S-convex}.
\end{itemize}
It is important to recall that, given $s_0\in(0,1)$ and $s\in(s_0,1)$, the constant $\bar\alpha>0$ depends on $s_0$, but not on~$s$.
In other words, $\bar\alpha$ stays positive as $s\uparrow1$.
\end{defi}

We establish the following.

\begin{thm}[Fully nonlinear operators]\label{th3}
Let $\I$ be any operator of the form \eqref{I}, with $L_a$ satisfying \eqref{L*1}-\eqref{L*2} for all $a\in\mathcal A$.
Let $\bar\alpha>0$ be given by Definition \ref{defi-baralpha}, let $\gamma\in(0,\bar\alpha)$, and let $\alpha\in(0,\bar\alpha)$ be such that $1+s+\alpha<2$.

Let $\varphi$ be any obstacle satisfying \eqref{obstacle}, $u$ be the solution to \eqref{obst-pb-fully}, and $d(x)={\rm dist}(x,\{u=\varphi\})$.
Then, at any free boundary point $x_0\in \partial\{u=\varphi\}$ we have
\begin{itemize}
\item[(i)] either\vspace{-2mm}
\[\qquad \qquad u(x)-\varphi(x)=c\,d^{1+s}(x)+o(|x-x_0|^{1+s+\alpha}),\qquad c>0,\]
\item[(ii)] or
\[\qquad\liminf_{r\downarrow0}\frac{\bigl|\{u=\varphi\}\cap B_r(x_0)\bigr|}{|B_r(x_0)|}=0 \quad \textrm{and}\quad u(x)-\varphi(x) =o(|x-x_0|^{\min(2s+\gamma,1+s+\alpha)}),\]\vspace{1mm}
\item[(iii)] or \hspace{30mm} $u(x)-\varphi(x) =o(|x-x_0|^{1+s+\alpha})$.
\end{itemize}
Moreover, the set of points $x_0$ satisfying (i) is an open subset of the free boundary and it is $C^{1,\gamma}$ for all $\gamma\in(0,\bar\alpha)$.
\end{thm}

Recall that the interior regularity for solutions to convex fully nonlinear equations is $C^{2s+\gamma}$; see \cite{CS3,S-convex}.
This is why for fully nonlinear operators we may have free boundary points satisfying (ii) above, in contrast with the case of linear operators (Theorem \ref{th2}).

Still, it is important to notice that when $s$ is close to 1 then we get the exact same result as in the linear case of Theorem \ref{th2}.
Indeed, in that case we have $2s+\bar\alpha\geq2$, thus we may take $\gamma$ such that $2s+\gamma\geq 1+s+\alpha$, and therefore all points (ii) satisfy (iii).

\begin{rem}[The class of kernels]
Notice that the $C^{1,s}$ regularity of solutions is very related to the class $\mathcal L_*$, and would not be true for more general classes of nonlocal operators.
Indeed, we studied in \cite{RS-K} the boundary regularity of solutions and showed that, while for the class $\mathcal L_*$ all solutions are $C^s$ up to the boundary, this is not true for fully nonlinear operators with more general kernels; see \cite[Section~2]{RS-K}.
This is why most of the results of the present paper are for the class $\mathcal L_*$.

Still, our techniques can be adapted to wider classes of kernels, such as the class $\mathcal L_0$ of \cite{CS}.
As explained above, in that case one does not expect solutions to be $C^{1,s}$, but a modification of our methods can be used to prove the $C^{1,\gamma}$ regularity of free boundaries in that case too.
We plan to do this in a future work.
\end{rem}

\subsection{Global strategy of the proof}

Let us briefly explain the global strategy of the proof of Theorem \ref{th2}.

We start with a free boundary point $x_0\in\partial\{u=\varphi\}$, and we assume that (b) in Theorem \ref{th2} does \emph{not} hold.
Then, the idea is to take a blow-up sequence of the type $v_r(x)=(u-\varphi)(x_0+rx)/\|u-\varphi\|_{B_r(x_0)}$.
However, we need to do it along an appropriate subsequence $r_k\to0$ so that the rescaled functions $v_{r_k}$ (and their gradients) have a ``good'' growth at infinity (uniform in $k$).
Once we do this, in the limit $r_k\to0$ we get a global solution $v_0$ to the obstacle problem, which is convex and has the following growth at infinity
\begin{equation}\label{growth-intro}
|\nabla v_0(x)|\leq C(1+|x|^{s+\alpha}),
\end{equation}
with $s+\alpha<\min\{1,2s\}$.
Such growth condition is very important in order to take limits $r_k\to0$ and to show that $v_0$ solves the obstacle problem.

The next step is to classify global \emph{convex} solutions $v_0$ to the obstacle problem with such growth.
We need to prove that the convex set $\Omega=\{v_0=0\}$ is a half-space.
For this, the first idea is to do a blow-down argument to get a new solution $\tilde v_0$, with the same growth \eqref{growth-intro}, and for which the set $\{\tilde v_0=0\}$ is a convex \emph{cone}~$\Sigma$.
Then, we separate into two cases, depending on the ``size'' of $\Sigma$.
If $\Sigma$ has zero measure, we show that $\tilde v_0$ would be a paraboloid in $\R^n$, which is a contradiction with the growth \eqref{growth-intro}.
On the other hand, if $\Sigma$ has nonempty interior, we first prove by a dimension reduction argument that $\Sigma$ is $C^1$ outside the origin.
Then, by convexity of $\tilde v_0$ we have a cone of directional derivatives satisfying $\partial_e \tilde v_0\geq0$ in $\R^n$.
Then, using a boundary Harnack estimate in $C^1$ domains \cite{RS-C1}, we prove that all such derivatives have to be equal (up to multiplicative constant) in $\R^n$, and thus that $\Sigma$ must be a half-space.
This implies that $\Omega$ was itself a half-space, and therefore $v_0$ is a 1D solution.

Once we have the classification of blow-ups, we show that the free boundary is Lipschitz in a neighborhood of $x_0$, and $C^1$ at that point.
This is done by adapting standard techniques from local obstacle problems to the present context of nonlocal operators.
Finally, by an appropriate barrier argument we show that the regular set is open, i.e., that all points in a neighborhood of $x_0$ do \emph{not} satisfy (b).
From here, we deduce that the free boundary is $C^1$ at every point in a neighborhood of $x_0$, and we show that this happens with a uniform modulus of continuity around $x_0$.
Finally, using again the results of \cite{RS-C1}, we deduce that the free boundary is $C^{1,\gamma}$ near $x_0$, and that $(u-\varphi)(x)=c_0d^{1+s}(x)+o(|x-x_0|^{1+s+\alpha})$ for some $c_0>0$.

\subsection{Plan of the paper}

The paper is organized as follows:
In Section \ref{sec2} we prove a $C^{1,\tau}$ estimate for solutions to the obstacle problem.
In Section \ref{sec3} we establish a uniqueness result for nonnegative solutions to linear equations in cones.
In Section \ref{sec4} we classify global convex solutions to the obstacle problem.
In Section \ref{sec5} we start the study of the free boundary at regular points.
Then, in Section \ref{sec6} we prove Theorem~\ref{th2}.
Finally, in Sections \ref{sec-fully-1} and \ref{sec-fully-2} we study the obstacle problem for fully nonlinear operators, and establish Theorem~\ref{th3}.

\section{$C^{1,\tau}$ regularity of solutions}
\label{sec2}

In this Section we provide some preliminary results and establish the $C^{1,\tau}$ regularity of solutions to the obstacle problem \eqref{obst-pb}-\eqref{obstacle}.
As we will see, the results of this section apply to more general operators of the form
\begin{equation}\label{L01}
Lu(x) = \int_{\R^n} \left(\frac{u(x+y)+u(x-y)}{2}-u(x)\right) \frac{\,b(y)}{|y|^{n+2s}}\,dy,
\end{equation}
with
\begin{equation} \label{L02}
b\in L^\infty(\R^n) \quad \mbox{satisfying} \quad b(y)=b(-y)\quad \mbox{and} \quad \lambda\le b\le \Lambda.
\end{equation}
This is the class $\mathcal L_0$ of \cite{CS}.
Thus, the kernels are not assumed to be homogeneous.

First, we show the following.

\begin{lem}[Semiconvexity]\label{semiconvex}
Let $L$ be any operator of the form \eqref{L01}-\eqref{L02}, $\varphi$ be any obstacle satisfying \eqref{obstacle}, and $u$ be the solution to \eqref{obst-pb}.
Then,
\begin{itemize}
\item[(a)] $u$ is semiconvex, with \[\hspace{40mm}\partial_{ee}u\ge -\|\varphi\|_{C^{1,1}(\R^n)}\qquad \textrm{for all}\ e\in S^{n-1}.\]
\item[(b)] $u$ is bounded, with \[\|u\|_{L^\infty(\R^n)}\leq \|\varphi\|_{L^\infty(\R^n)}.\]
\item[(c)] $u$ is Lipschitz, with \[\|u\|_{\rm Lip(\R^n)}\leq \|\varphi\|_{C^{0,1}(\R^n)}.\]
\item[(d)] $Lu$ is bounded, with \[\|Lu\|_{L^\infty(\R^n)}\leq C.\]
\end{itemize}
The constant $C$ depends only on $\|\varphi\|_{C^{1,1}(\R^n)}$ and ellipticity constants.
\end{lem}

\begin{proof}
The proofs are essentially the same as the ones in \cite{S-obst}.

(a) First, by definition $u$ is the least supersolution which is above the obstacle $\varphi$ and is nonnegative at infinity.
Namely, if $v$ satisfies $-Lv\geq0$ in $\R^n$, $v\geq\varphi$ in $\R^n$, and $\liminf_{|x|\rightarrow \infty}v(x)\geq0$, then $v\geq u$.

Thus, for any given $h\in \R^n$ we may take
\[v(x)=\frac{u(x+h)+u(x-h)}{2}+C|h|^2.\]
This function clearly satisfies $-Lv\geq0$ in $\R^n$, and also $v\geq\varphi$ in $\R^n$ for $C=\|\varphi\|_{C^{1,1}(\R^n)}$.
Hence, we have $v\geq u$ in $\R^n$, and therefore
\[\frac{u(x+h)+u(x-h)-2u(x)}{|h|^2}\geq -C.\]
Since $C$ is independent of $h\in B_1$, we get $\partial_{ee}u\geq -C$ for all $e\in S^{n-1}$.

(b) A similar argument with the constant function $v(x)=\|\varphi\|_{L^\infty(\R^n)}$ leads to the bound $\|u\|_{L^\infty(\R^n)}\leq \|\varphi\|_{L^\infty(\R^n)}$.

(c) Taking now the function $v(x)=u(x+h)+C|h|$, we find the estimate $\|u\|_{\rm Lip(\R^n)}\leq \|\varphi\|_{\rm Lip(\R^n)}$.

(d) By definition of $u$, we have $-Lu\geq0$ in $\R^n$.
On the other hand, by the semiconvexity of $u$ and since $u\in L^\infty$ we have $Lu\geq -C$.
Hence, the $L^\infty$ bound for $Lu$ follows.
\end{proof}

The next Proposition can be applied to $u-\varphi$, where $u$ is the solution to \eqref{obstacle}.

If implies that all solutions of the obstacle problem implies that all  are $C^{1,\tau}$ for some $\tau>0$ (even with $L$ in the ellipticity class $\mathcal L_0$).

\begin{prop}\label{first-regularity-u}
Let $L$ be any operator of the form \eqref{L01}-\eqref{L02}, and $u\in{\rm Lip}(\R^n)$ be any function satisfying, for all $h\in\R^n$ and $e\in S^{n-1}$,
\[\begin{array}{rcll}
u&\geq&0\quad &\textrm{in}\ \R^n \\
\partial_{ee}u&\geq& -K\quad &\textrm{in}\ B_2 \\
L(u-u(\cdot-h))&\geq& -K|h|\quad &\textrm{in}\ \{u>0\}\cap B_1\\
|\nabla u|&\leq& K(1+|x|^{s+\alpha}) \quad &\textrm{in}\ \R^n.
\end{array}\]
Then, there exists a small constant $\tau>0$ such that
\[\|u\|_{C^{1,\tau}(B_{1/2})}\leq CK.\]
The constants $\tau$ and $C$ depend only on $n$, $s$, $\alpha$, $\lambda$, and $\Lambda$.
\end{prop}

The Proposition will follow from the following result.
Recall that $M^+_{\mathcal L_0}$ denotes the extremal operator associated to the class $\mathcal L_0$, i.e.,
\[M^+_{\mathcal L_0}u=\sup_{L\in \mathcal L_0}Lu.\]
Here, $\mathcal L_0$ is the class of all operators of the form \eqref{L01}-\eqref{L02}.

\begin{lem}\label{lemC1+tau}
There exist constants $\tau>0$ and $\delta>0$ such that the following statement holds true.

Let $u\in {\rm Lip}(\R^n)$ be a solution to
\[\begin{array}{rcll}
u&\geq&0\quad &\textrm{in}\ \R^n \\
\partial_{ee}u&\geq& -\delta\quad &\textrm{in}\ B_2 \quad \mbox{for all }e\in S^{n-1} \\
M^+_{\mathcal L_0}(u- u(\cdot-h))&\geq& -\delta |h| \quad &\mbox{in } \{u>0\}\cap B_2 \quad \mbox{for all }h\in\R^n,
\end{array}\]
satisfying the growth condition
\[  \sup_{B_R} |\nabla u| \le R^\tau \quad \mbox{for }R\geq1.\]
Assume that $u(0)=0$.
Then,
\[ |\nabla u(x)| \le 2|x|^{\tau}. \] 
The constants $\tau$ and $\delta$ depend only on $n$, $s$, and ellipticity constants.
\end{lem}

\begin{proof}
We define
\[
\theta(r) := \sup_{r'\ge r}  (r')^{-\tau}\sup_{B_{r'}}  |\nabla u|
\]

Note that $\theta(r)\le 1$ for $r\ge 1$ by the growth control.
We will prove that $\theta(r) \le 2$ for all $r\in (0,1)$, and this will yield the desired result.

Assume by contradiction that $\theta(r)> 2$ for some $r\in(0,1)$.
Then, by definition of $\theta$, there will be $r'\in(r,1)$ such that
\[ (r')^{-\tau}\sup_{B_{r'}}  |\nabla u| \ge (1-\epsilon) \theta(r) \ge (1-\epsilon)\theta(r') \ge \frac{3}{2},\]
where $\epsilon>0$ is a small number to be chosen later. Here where we used that $\theta$ is nonincreasing.

We next define
\[\bar u(x) := \frac{u(r' x)}{\theta(r') (r')^{1+\tau}}.\]
Since $\tau\in (0,s)$, rescaled function satisfies
\[\begin{array}{rcll}
\bar u&\geq&0\quad &\textrm{in}\ \R^n \\
D^2 \bar u&\geq& -(r')^{2-1-\tau}\delta \ge -\delta \quad &\textrm{in}\ B_{2/r'} \supset B_2 \\
M^+_{\mathcal L_0}(\bar u- \bar u(\cdot-\bar h))&\geq& - (r')^{2s-1-\tau}\delta |r'\bar h| \geq -\delta |\bar h| \quad &\mbox{in } \{\bar u>0\}\cap B_2 \quad \mbox{for all }h\in\R^n,
\end{array}\]
Moreover, by definition of $\theta$ and $r'$, the rescaled function $\bar u$ satisfies
\begin{equation}\label{growthc11}
1-\epsilon \le  \sup_{|\bar h|\le 1/4} \sup_{B_1} \frac{\bar u- \bar u(\cdot-\bar h)}{|\bar h|}  \quad \mbox{and}\quad
 \sup_{|\bar h|\le 1/4} \sup_{B_R} \frac{\bar u-\bar u(\cdot-\bar h)}{|\bar h|} \le  (R+1/4)^\tau
 \end{equation}
for all $R\ge 1$.

Let $\eta\in C^2_c(B_{3/2})$ with $\eta\equiv 1$ in $B_1$ and $\eta \le 1$ in $B_{3/2}$.
Then,
\[\sup_{|\bar h|\le 1/4}  \sup_{B_{3/2}}  \left(\frac{\bar u- \bar u(\cdot-\bar h)}{|\bar h|} +3\epsilon\eta\right)\geq 1+2\epsilon.\]
Fix $h_0\in B_{1/4}$ such that
\[ t_0 := \max_{B_{3/2}}  \left(\frac{\bar u- \bar u(\cdot- h_0)}{| h_0|} +3\epsilon\eta\right)\geq 1+\epsilon.\]
and let  $x_0\in B_{3/2}$ be such that
\begin{equation}\label{etatouches}
\frac{\bar u(x_0)- \bar u(x_0- h_0)}{| h_0|} + 3\epsilon \eta(x_0) = t_0 .
\end{equation}

Let us denote
\[ v: = \frac{\bar u- \bar u(\,\cdot\,- h_0)}{| h_0|} . \]

Then, we have
\[v + 3\epsilon \eta \le v(x_0) +3\epsilon\eta(x_0) =  t_0  \quad \mbox{in }\overline{B_{3/2}}.\]
Moreover, if $\tau$ is taken small enough then
\[\sup_{B_4} v \leq (4+1/4)^\tau<1+\epsilon\leq t_0,\]
and therefore
\begin{equation}\label{etatouches2}
v +3\epsilon\eta\le t_0\quad \mbox{in }\overline{B_2}.
\end{equation}
Note also that $x_0\in \{\bar u>0\}$ since otherwise $\bar u(x_0)-\bar u(x_0-h_0)$ would be a nonpositive number.

We now evaluate the equation for $v$ at $x_0$ to obtain a contradiction.

Now we crucially use that $D^2\bar u \ge -\delta {\rm Id}$, $\bar u \ge 0$, and $\bar u(0)=0$.
It follows that, for $z\in B_2$ and $t'\in(0,1)$,
\[\bar u(t'z) \le t' \bar u(z)+ (1-t') \bar u(0)  +  \frac {\delta |z|^2} 2  t'(1-t') \le \bar u(z) + \frac {\delta |z|^2} 2  t'(1-t') \]
and thus, for $t\in(0,1)$, setting $z =x(1+ t/|x|)$ and $t' = 1/(1+t/|x|)$ we obtain, for $x\in B_1$,
\[\bar u(x) -\bar u\left( x + t\frac{x}{|x|}\right)  \le \frac \delta 2 (|x|+t)^2 \frac{t/|x|}{(1+t/|x|)^2 } = \frac {\delta |x| t} 2 \le \delta t .\]

Therefore, denoting $e = h_0 /|h_0|$, $t = |h_0|\le 1$ and using that by \eqref{growthc11}
\[ \|\bar u\|_{\rm Lip(B_1)} \le 3/2\]
we obtain
\[\begin{split}
v(x)= \frac{\bar u(x)- \bar u(x- te)}{t}
&\le  \frac{\bar u(x)- \bar u(x- te)}{t}  +  \frac{\bar u\left( x + t\frac{x}{|x|}\right) -\bar u(x) }{t}  + \delta
\\
&\le \frac{\bar u\left( x + t\frac{x}{|x|}\right)- \bar u(x- te)}{t}  +   \delta
\\
&\le  \frac32\left| e +\frac{x}{|x|} \right|+\delta \,\le\, \frac{1} 4
\end{split}
\]
in the set
\[ \mathcal C_e : = \left\{x\,:\, \left|e+ \frac{x}{|x|}\right| \le \frac 18 \right\}\cap B_1,\]
provided $\delta$ is taken smaller than $1/16$.

Then, recalling that
\[\begin{split}
M^+_{\mathcal L_0}v(x_0) &= \Lambda \int_{\R^n} \left(\frac{v(x_0+y)+v(x_0-y)}{2}-v(x_0)\right)_+|y|^{-n-2s}\,dy \\
&\qquad\qquad- \lambda \int_{\R^n} \left(\frac{v(x_0+y)+v(x_0-y)}{2}-v(x_0)\right)_-|y|^{-n-2s}\,dy,
\end{split}\]
we want to show that $M^+_{\mathcal L_0}v(x_0)<-\delta$ in order to get a contradiction.
Indeed, using
\[1-2\epsilon\leq v(x_0)\leq 1+\epsilon\]
and
\[\frac{v(x_0+y)+v(x_0-y)}{2}-v(x_0)\leq \left\{\begin{array}{ll}
C\epsilon |y|^2&\quad\textrm{in}\quad B_2\\
(|y|+2)^\tau-1+2\epsilon&\quad\textrm{in}\quad \R^n\setminus B_1\\
1/4-1+2\epsilon&\quad\textrm{in}\quad (-x_0+\mathcal C_e\cap B_1),
\end{array}\right.\]
we find
\[\begin{split}
M^+_{\mathcal L_0}v(x_0) &\le \Lambda \int_{B_1} C\epsilon\, |y|^2\, |y|^{-n-2s}\,dy + \Lambda\int_{\R^n\setminus B_1} \bigl\{(|y|+2)^\tau-1+2\epsilon\bigr\} |y|^{-n-2s}\,dy\\
&\qquad\qquad\qquad - \frac{\lambda}{2} \int_{-x_0+\mathcal C_e\cap B_1} \bigl( 1-2\epsilon-1/4\bigr)|y|^{-n-2s}\,dy \\
&\le  C\epsilon+C\int_{\R^n\setminus B_1} \bigl\{(|y|+2)^\tau-1\bigr\}|y|^{-n-2s}\,dy-c,
\end{split}\]
with $c>0$ independent of $\delta$ and $\tau$ (if $\tau$ is small enough).
Thus, if $\epsilon$ and $\tau$ are taken small enough we obtain $-\delta \le M^+_{\mathcal L_0}v(x_0)\le -c/2$; a contradiction  when $\delta\le c/4$.
\end{proof}

We finally give the proof of Proposition \ref{first-regularity-u}.
In fact, the exact same proof will yield the following result, which is an extension of Proposition \ref{first-regularity-u} to equations with bounded measurable coefficients.
This will be used in Sections \ref{sec-fully-1} and \ref{sec-fully-2}.

\begin{prop}\label{first-regularity-u-fully}
Let $L$ be any operator of the form \eqref{L01}-\eqref{L02}, and $u\in{\rm Lip}(\R^n)$ be any function satisfying
\[\begin{array}{rcll}
u&\geq&0\quad &\textrm{in}\ \R^n \\
\partial_{ee}u&\geq& -K\quad &\textrm{in}\ B_2 \\
M_{\mathcal L_0}^+(u-u(\cdot-h))&\geq& -K|h|\quad &\textrm{in}\ \{u>0\}\cap B_1\\
|\nabla u|&\leq& K(1+|x|^{s+\alpha}) \quad &\textrm{in}\ \R^n
\end{array}\]
for all $h\in \R^n$ and $e\in S^{n-1}$.
Then, there exists a small constant $\tau>0$ such that
\[\|u\|_{C^{1,\tau}(B_{1/2})}\leq CK.\]
The constants $\tau$ and $C$ depend only on $n$, $s$, $\lambda$, and $\Lambda$.
\end{prop}

\begin{proof}[Proof of Propositions \ref{first-regularity-u} and \ref{first-regularity-u-fully}]
Let $x_0\in B_{3/4}$ be any point at which $u(x_0)=0$.
Then, after rescaling and truncating the function $u$, we may apply Lemma \ref{lemC1+tau} to find
\begin{equation}\label{case1j}
|\nabla u(x)|\leq CK|x-x_0|^{\tau}
\end{equation}
at every such point $x_0\in B_{3/4}$.

Now, let $x\in\{u>0\}$, let $x_0$ be its nearest point on $\{u=0\}$, and let $r=|x-x_0|$.
By \eqref{case1j}, we have
\[|\nabla u(x+rz)|\leq CKr^\tau(1+|z|^\tau).\]
Therefore, by interior regularity estimates \cite{CS}, we will have
\begin{equation}\label{case2j}
[\nabla u]_{C^\tau(B_{r/2}(x))}\leq CK
\end{equation}
for all $x\in \{u>0\}\cap B_{3/4}$.

Let now $x$ and $y$ be any two points in $B_{1/2}$, and let us show that
\begin{equation}\label{we-want-j}
|\nabla u(x)-\nabla u(y)|\leq C|x-y|^\tau.
\end{equation}
Let $d(z)={\rm dist}(z,\{u=0\})$, and $r=\min\{d(x),d(y)\}$.
Define also $R=|x-y|$.

If $2R\geq r$, then \eqref{we-want-j} follows from \eqref{case1j} and the triangle inequality.
If $2R<r$ then $B_{2R}(x)\subset\{u>0\}$, and thus \eqref{we-want-j} follows from \eqref{case2j}.
Hence \eqref{we-want-j} holds, and therefore
\[\|\nabla u\|_{C^\tau(B_{1/2})}\leq C,\]
as desired.
\end{proof}

\section{Uniqueness of positive solutions to $Lv=0$ in $C^1$ cones}
\label{sec3}

The aim of this section is to prove the following Phragmen-Lindel\"of type result.

\begin{thm}\label{thm-uniqueness}
Let $\Sigma\subset \R^n$ be any cone with nonempty interior, with vertex at $0$, and such that $\partial \Sigma$ is $C^1$ away from $0$.
Let $L\in \mathcal L_*$, and $u_1,u_2$ be functions in $C(\R^n)$ satisfying
\[\int_{\R^n}  u_i(y) (1+|y|)^{-n-2s}\,dy  < \infty.\]
Assume that $u_i$ are viscosity solutions to
\begin{equation}\label{eqn1space}\begin{cases}
Lu_i = 0  \quad &\mbox{in } \R^n\setminus \Sigma\\
u_i=  0 & \mbox{in }\Sigma\\
u_i> 0 & \mbox{in } \R^n \setminus \Sigma.
\end{cases}\end{equation}
Then,
\[ u_1 \equiv K u_2 \quad \mbox{ in } \R^n\]
for some $K>0$.
\end{thm}

The proof of the previous result requires several ingredients.
First, we will need a boundary Harnack inequality for the class $\mathcal L_*$ in $C^1$ domains, established in \cite{RS-C1}.

\begin{prop}[\cite{RS-C1}]\label{bdryHarnackC1}
Let $L\in \mathcal L_*$, and $\Omega\subset \R^n$ be a $C^1$ domain with modulus of continuity $\rho$.
Let $u_1,u_2\in C(\R^n)$ be two (viscosity) solutions of
\begin{equation}\label{eqnpointbdryreg}
\begin{cases}
Lu_i = 0  \quad &\mbox{in } B_1\cap \Omega\\
u_i =  0 & \mbox{in }B_1\setminus \Omega\\
u_i \geq  0 & \mbox{in }\R^n.
\end{cases}
\end{equation}
Assume
\[\int_{\R^n} u_i(y) \bigl(1+|y|\bigr)^{-n-2s}\,dx=1.\]

Then,
\[ 0<c\leq \frac{u_1}{u_2}\leq C\qquad \textrm{in}\ B_{1/2},\]
where $c$ and $C$ depend only on $\rho$, $n$, $s$, and ellipticity constants.
\end{prop}

We next show an auxiliary lemma, a version of boundary Harnack, which is the first step towards Theorem \ref{thm-uniqueness}.

\begin{lem}\label{lemB1}
Let $\Sigma\subset \R^n$ be a cone with nonempty interior, with vertex at $0$, and such that $\partial \Sigma$ is $C^1$ away from $0$.
Let $L\in \mathcal L_*$, and $u_i$, $i=1,2$, be two solutions of
\begin{equation}\label{eqn1}
\begin{cases}
Lu_i = 0  \quad &\mbox{in } B_2\setminus \Sigma\\
u_i =  0 & \mbox{in }\Sigma\\
u_i > 0 & \mbox{in } \R^n \setminus \Sigma.
\end{cases}
\end{equation}
Assume in addition that
\begin{equation}\label{integral1}
\int_{\R^n}  u_i(y) (1+|y|)^{-n-2s}\,dy  = 1.
\end{equation}
Then,
\[ u_1 \ge c\, u_2 \quad \mbox{and} \quad u_2 \ge c\, u_1 \quad \mbox{ in }\overline{B_1}\]
for some $c>0$ depending only on $n$, $s$, $\Sigma$, and the ellipticity constants.
\end{lem}

\begin{proof} We divide the proof into three steps.

{\em Step 1.} We show first that if $P$ is a point with $|P|=1$ and $B_1(P)\subset \R^n \setminus \Sigma$ then \eqref{eqn1} and \eqref{integral1} imply that $u_1$ and $u_2$ are both comparable to $1$ in $\overline{B_{1/4}(P)}$.
First, by Theorem 5.1 in \cite{CS3} we have that
\[  0 \le u_i \le C  \quad \mbox{ in } \overline{B_{1/4}(P)}.\]
On the other hand we have the following dichotomy: either
\begin{itemize}
\item[(a)] we have $\int_{{B_{3/4}(P)}} u_i(y) (1+|y|)^{-n-2s}\,dy \ge \frac12$.
\end{itemize}
or
\begin{itemize}
\item[(b)] we have $\int_{\R^n \setminus {B_{3/4}(P)}} u_i(y) (1+|y|)^{-n-2s}\,dy \ge \frac12$.
\end{itemize}
In both cases we claim that
\begin{equation}\label{boundbybelow}
 u_i \ge c >0  \quad \mbox{ in } \overline{B_{1/4}(P)}.
\end{equation}
Indeed, in the case (a) the supremum of $u_i$ in $B_{1/4}(P)$ is bounded below by a positive universal constant and hence the interior Harnack inequality immediately implies \eqref{boundbybelow}.

In case (b) the function
\[\phi(x) = c\bigl(1-4|x-P|\bigr)^2 \chi_{B_{3/4}(P)}(x)+ u_i(x) \chi_{\R^n\setminus B_{3/4}(P)}(x),\]
with $c$ small enough, satisfies $L\phi\ge0$ in $\overline B_{1/2}$ (it is a subsolution of the equation).
Therefore, since $u\ge \phi$ in $\R^n \setminus B_{1/2}(P)$ we obtain
\[  u_i(x) \ge  c\bigl(1-4|x-P|)^2 \quad \mbox{in } B_{1/2}(P),\]
and \eqref{boundbybelow} follows.

{\em Step 2.}  We show that $u_1$ and $u_2$ are comparable up to the boundary in the annulus
\[A:= B_{3/2}\setminus B_{1/2}.\]

Note that by assumption the domain $B_{2}\setminus\Sigma$ will be $C^1$ at all the boundary points $z \in \partial \Sigma\cap (B_{3/2}\setminus B_{1/2})$.

Therefore, by Proposition \ref{bdryHarnackC1}, we have
\begin{equation}\label{bdryexpansion}
 0<c\leq \frac{u_1}{u_2}\leq C\quad \textrm{in}\ A
\end{equation}
with $C$ depending only on $n$, $s$, $\Sigma$ and the ellipticity constants.

{\em Step 3.}  We finally show that $u_1\ge c\, u_2$.
Let us define
\[ w = \bigl\{u_2 \chi_{B_1}  + C\chi_{B_{1/4}(P)}\bigl\}.\]
It follows from \eqref{integral1} that $w$ is a subsolution in $B_{1/2}\setminus \Sigma$, i.e.,
\[ Lw \ge 0  \quad \mbox{in }B_{1/2}\setminus \Sigma,\]
provided that $C$ is chosen large depending only on $n$, $s$, and ellipticity constants.
Notice that here we are exploiting the nonlocal character of the equation in order to obtain such a simple subsolution.

Thanks to Step 2 we have $u_1 \ge c\, w$ in $B_{3/2}\setminus B_{1/2}$, for some $c>0$ depending only on $n$, $s$, $\Sigma$ and the ellipticity constants.
Since $w=0$ outside $B_{3/2}$ and $u_1\ge 0$ we also have $u_1 \ge c w$ outside $B_{3/2}$.
The same inequality trivially holds in $\Sigma$ where both $u_1$ and $w$ vanish.
Thus, it follows from the maximum principle that
\[ u_1 \ge c\, w \ge c\,u_2 \quad \mbox{in all of}\ B_1,\]
as desired.
\end{proof}

Using the previous Lemma, we can now establish Theorem \ref{thm-uniqueness}.

\begin{proof}[Proof of Theorem \ref{thm-uniqueness}]
Let $P$ be a point with $|P|=1$ and $B_1(P)\subset \R^n \setminus \Sigma$.
We may assume after normalization that
\[ u_i(P) = 1,\]
and we want to prove $u_1 \equiv u_2$.

{\em Step 1}. We first show, using Lemma \ref{lemB1} at every scale, that
\begin{equation}\label{comparable}
u_1\ge c\,u_2  \quad \mbox{and}\quad u_2\ge cu_1  \quad \mbox{in all of }\R^n,
\end{equation}
for some $c>0$ depending only on $n$, $s$, ellipticity constants, and $\Sigma$.

Indeed, given $R\ge 1$ we define $\bar u_i$ as
\[\bar u_i (x)  = \frac{u_i(Rx)}{C_i},\]
where $C_i$ are chosen so that
\[\int_{\R^n}  \bar u_i(y) (1+|y|)^{-n-2s}\,dy  = 1.\]
By Lemma \ref{lemB1}, we have
\begin{equation}\label{comparableB1}
\bar u_1 \ge  c\, \bar u_2\quad \mbox{and}\quad \bar u_2 \ge  c\,\bar u_1\quad \mbox{in }B_{1/2}.
\end{equation}
But since
\[  1 = u_i(P)  =  C_i \bar u_i(P/R),\]
and since $\bar u_1(P/R)$ and $\bar u_2(P/R)$ are comparable, then we obtain that $C_1$ and $C_2$ are comparable.
Hence, rescaling the first inequality in \eqref{comparableB1} from $B_1$ to $B_R$ we obtain
\[  u_1\ge c\,u_2  \quad \mbox{and}\quad u_2\ge c\,u_1   \quad \mbox{in  }B_R.\]
Since $R\ge 1$ is arbitrary, \eqref{comparable} follows.

{\em Step 2}.
We define
\[ \bar c = \sup\bigl\{ c>0\ :\ u_2\ge c\,u_1\mbox{ in all of }\R^n\bigr\}.\]
By Step 1 we also have $\bar c \neq \pm \infty$.
Define
\[ v = u_2-\bar c\, u_1,\]
which is either $0$ in all of $\R^n$ or positive in $\R^n\setminus \Sigma$ (by the interior Harnack inequality).
If $v>0$ in $\R^n\setminus \Sigma$, applying Step 1 to the two functions $v/v(P)$ and $u_1$, with $v/v(P)$ playing the role of $u_2$, we deduce that $v>\delta u_1$ for some $\delta>0$ ---which may depend on $v$.
This is a contradiction with the definition of $\bar c$, and hence it must be $\bar v\equiv 0$ and $u_2 \equiv \bar c u_1$.
Since $u_1(P)=u_2(P)=1$, then $\bar c=1$, and the result is proved.
\end{proof}

\section{Classification of global convex solutions}
\label{sec4}

The aim of this section is to prove the following result, which classifies all global convex solutions to the obstacle problem under a growth assumption on $u$.

\begin{thm}\label{thmclassif}
Let $\Omega\subset \R^n$ be a closed convex set, with $0 \in \partial\Omega$.
Let $\alpha\in(0,s)$ be such that $1+s+\alpha<2$.
Let $u\in C^{1}(\R^n)$ be a function satisfying, for all $h\in\R^n$,
\begin{equation}\label{eqnliouville}\begin{cases}
L (\nabla u) = 0  \quad &\mbox{in } \R^n\setminus \Omega\\
L \bigl(u-u(\cdot-h)\bigr) \geq 0  \quad &\mbox{in } \R^n\setminus \Omega\\
D^2 u \ge 0 \quad &\mbox{in } \R^n\\
u=  0 & \mbox{in }\Omega\\
u\ge 0 & \mbox{in } \R^n.
\end{cases}\end{equation}
Assume that $u$ satisfies the following growth control
\begin{equation}\label{growthctrolgradient}
\|\nabla u\|_{L^\infty(B_R)} \le R^{s+\alpha} \quad \mbox{ for all }R\ge 1.
\end{equation}
Then, either $u \equiv 0$, or
\[ \Omega =\{e\cdot x\le 0\}\quad \mbox{ and }\quad u(x) = C (e\cdot x)_+^{1+s} \]
for some $e\in S^{n-1}$ and $C>0$.
\end{thm}

We will establish Theorem \ref{thmclassif} by using a blow-down argument combined with the following Proposition, which corresponds to the particular case in which $\Omega$ is a convex cone $\Sigma$ with nonempty interior.

\begin{prop}\label{propclassif}
Let $\Sigma$ be a closed convex cone in $\R^n$ with nonempty interior and vertex at $0$.
Then, Theorem \ref{thmclassif} holds for $\Omega = \Sigma$.
\end{prop}

To prove the Proposition we will need the following.

\begin{lem}\label{convex-1D}
Let $u: \R^n \rightarrow \R$ be a convex function such that the set $\{u=0\}$ contains the straight line $\{te'\,:\, t\in \R\}$,  $e'\in S^{n-1}$.
Then, $u(x+te')=u(x)$ for all $x\in \R^n$ and all $t\in\R$.
\end{lem}

\begin{proof}
Let $p(x)= ax+b$ be a supporting hyperplane of the epigraph of $u$. Since  $\{u=0\}$ contains a straight line parallel to $e$ it must be $a\cdot e'=0$, since otherwise
\[0 = u(te') \ge p(te') =  t(a\cdot e') +b\]
would be violated by taking $t = C (a\cdot e') $, with $C>0$ large. Thus, every vector belonging to the sub-differential of $u$ at some point in $\R^n$ is orthogonal to $e'$, and this means that $u$ is constant on lines that are parallel to $e'$.
\end{proof}

We give now the:

\begin{proof}[Proof of Proposition \ref{propclassif}]
We prove it by induction on the dimension $n$.
We divide it into two steps.

{\em Step 1}. We show first that the proposition holds when $\Sigma$ is $C^1$ away from $0$  ---in particular for the case $n=2$.

Since $\Sigma$ has nonempty interior we may choose $n$ linearly independent vectors $e_i$, $i=1,2,...,n$, such that $-e_i \in \Sigma$ and $|e_i|=1$.
Then, we consider
\[ v_i = \partial_{e_i} u,\]
and notice that they solve
\begin{equation}\begin{cases}
L v_i = 0  \quad &\mbox{in } \R^n\setminus \Sigma\\
v_i=  0 & \mbox{in }\Sigma\\
v_i\ge 0 & \mbox{in } \R^n.
\end{cases}\end{equation}
The non-negativity condition $v_i\ge0$ follows from the convexity $D^2 u \ge 0$ and the fact that $-e_i\in \Sigma = \{u=0\}$.

Now, since $\Sigma$ is $C^1$ away from $0$, it follows from Theorem \ref{thm-uniqueness} that the linear space of functions
\[ \left\{ \sum_{i=1}^n \lambda_i v_i \ : \ (\lambda_1, \lambda_2, \dots,\lambda_n) \in \R^n\right\} \]
has dimension at most one.
This means that $v_i=\alpha_iv_k$ for some $1\leq k\leq n$ and for all $i=1,...,n$, and thus $\partial_{e_i-\alpha_ie_k}u\equiv0$ in $\R^n$ for all $i\neq k$.
It follows that $u$ has $1-D$ symmetry, that is, $u= \phi(e\cdot x)$ for some $e\in S^{n-1}$ and $\phi:\R\rightarrow \R$.
In particular, $\Sigma = \{e\cdot x\le 0\}$, where we have used that 0$\in \partial \Sigma$.

But then $\phi\in C^1$ is a nonnegative solution of $L (\phi')=0$ in $\R_+$, $\phi=0$ in $\R^-$, satisfying $\phi'(t)\leq C(1+t^{s+\alpha})$, with $\alpha\in(0,s)$.
It follows from Theorem 4.1 in \cite{RS-stable} that $\phi'(t)=K(t)_+^s$, and thus $\phi(t)=K(t)_+^{1+s}$ for some $K\geq0$.
Hence,
\[ u(x) = K(e\cdot x)_+^{1+s},\]
as desired.

{\em Step 2}. We show next by induction on the dimension that the cone $\Sigma$ will be $C^1$ away from $0$, and hence that we can always apply Step 1.

Assume that the statement of the proposition holds true up to dimension $n-1$.
Then, we will prove that convex cone $\Sigma\subset\R^n$ must be $C^1$ away from $0$.
More precisely, we will show that for any $z\in \partial\Sigma\cap \partial B_1$ the blow-up of $\Sigma$ at the point $z$ is a half-space. This, together with the fact that $\Sigma$ is convex will imply that $\partial\Sigma$ is $C^1$ away from $0$.

Let us consider a blow-up sequence at points $z\in \partial B_1\cap \partial\Sigma$.
For $r>0$, we define
\[ \theta(r)  = \sup_{r'\ge r} \frac{ \|\nabla u\|_{L^\infty(B_{r'}(z))}}{(r')^{s+\alpha}}.\]
Note that $\theta(r) < \infty$ for all $r>0$ thanks to the growth control \eqref{growthctrolgradient}.

Moreover, we claim that
\begin{equation}\label{claim-theta-infty}
\theta(r)\rightarrow\infty\quad \textrm{as}\quad r\rightarrow0.
\end{equation}
Indeed, let $B$ be a ball of radius 1 such that $B\subset\R^n\setminus\Sigma$ and $z\in \partial B$ (recall $\Sigma$ is convex), and let $w_0$ be the solution of $Lw_0=-1$ in $B$, $w_0=0$ in $\R^n\setminus B$.
By the results of \cite{RS-K}, we have that $w_0\geq cd^s$ for some $c>0$, where $d(x)={\rm dist}(x,\R^n\setminus B)$.
Let $K\subset\subset B$ be any compact set in $B$, and $\eta\in C^\infty_c(K)$ be such that $0\leq \eta\leq1$ and $\int_K\eta>0$.
Then, the function $\phi=w_0+C\eta$ satisfies $L\phi\geq0$ in $B\setminus K$, $\phi\equiv0$ in $\R^n\setminus B$, and $\phi\leq C$ in $K$.

Let now $e\in S^{n-1}$ be such that $-e\in \Sigma$.
Then, $\partial_e u\geq0$ in $\R^n$, and by the Harnack inequality $\partial_e u\geq c>0$ in $K$.
Therefore, we may use $\epsilon\phi$ as a subsolution to find that $\partial_e u\geq\epsilon\phi$ in $B$, and in particular $\partial_eu \geq c d^s$ in $B$ for some small constant $c>0$.
Hence, $\|\nabla u\|_{L^\infty(B_r(z))}\geq cr^s$ for all $r\in (0,1)$, and this yields \eqref{claim-theta-infty}.

Now, thanks to \eqref{claim-theta-infty}, for all $m\in \mathbb N $ there  are $r'_m\ge 1/m$ and $z_m\in \partial B_1\cap \partial\Sigma$ such that $r_m'\rightarrow0$ and
\[ \ (r'_m)^{-s-\alpha} \|\nabla u\|_{L^\infty(B_{r'_m}(z))} \ge \frac{\theta(1/m)}{2} \ge \frac{\theta(r'_m)}{2}.\]
Then the blow-up sequence
\[ u_m(x) := \frac{u(z+r'_m x)}{(r'_k)^{1+s+\alpha}\theta(r'_k)}\]
satisfies the growth control
\[   \|\nabla u_m\|_{L^\infty(B_R)} \le R^{s+\alpha}\quad \mbox{for all }R\ge 1 \]
and the ``nondegeneracy'' condition
\[  \|\nabla u_m\|_{L^\infty(B_1)}  \ge \frac 12.   \]

By the $C^{1,\tau}$ estimates of Proposition \ref{first-regularity-u} and the Arzel\`a-Ascoli theorem, the functions $u_m$ converge (up to a subsequence) locally uniformly in the $C^1$ topology to a function $u_\infty$ that
satisfies
\begin{equation}\label{growthuinfty1}
\|\nabla u_\infty\|_{L^\infty(B_{R})} \le R^{s+\alpha}\quad \mbox{for all }R\ge 1,
\end{equation}
\begin{equation}\label{gnondeguinfty1}
\|\nabla u_\infty\|_{L^\infty(B_{1})}  \ge \frac 12,
\end{equation}
and
\begin{equation}\begin{cases}
L (\nabla u_\infty) = 0  \quad &\mbox{in } \R^n\setminus \Sigma_\infty\\
L \bigl(u_\infty-u_\infty(\cdot-h)\bigr) \geq 0  \quad &\mbox{in } \R^n\setminus \Sigma_\infty\\
D^2 u_\infty \ge 0 \quad &\mbox{in } \R^n\\
u_\infty = 0 &\mbox{in }\Sigma_\infty,
\end{cases}\end{equation}
where $\Sigma_\infty$ is the blow-up of $\Sigma$ at $z\in \partial\Sigma\cap \partial B_1$.

Now, since $\Sigma$ is a cone with vertex at $0$ and $|z|=1$, the cone $\Sigma_\infty$ will satisfy
\[  \lambda e'+\Sigma_\infty = \Sigma_\infty \quad \mbox{for all }\lambda\in \R,\]
at least for one vector $e'\in S^{n-1}$ (just take $e'=z$).

But since $u_\infty\ge 0$ is convex and its zero level set $\Sigma_\infty$ is invariant under translations in the direction $e'$, then by Lemma \ref{convex-1D}
\[ u_\infty(\lambda e' +\,\cdot\,) \equiv u_\infty.\]
Thus, $u_\infty$ is a function of only $n-1$ affine variables and hence solves the same problem in dimension $n-1$.

It then follows from the induction hypothesis that $\Sigma_\infty$ is a half-space and that $u_\infty= K(e\cdot x)_+^{1+s}$ for some $e\in S^{n-1}$ and $K>0$ ---the fact that $K$ is not zero follows from \eqref{gnondeguinfty1}.

Thus, $\Sigma$ is a convex cone ,with vertex at $0$, with nonempty interior, and such that its blow up at every point $z\in \partial \Sigma$ with $1/2\le |z|\le 3/2$ is a half-space. Therefore, $\partial \Sigma$ is $C^1$ away from the origin. Indeed, since $\Sigma$ is convex, $\partial\Sigma$ is a convex graph locally. Namely, for some $\delta>0$ we have under the appropriate choice of coordinates,
\[ \Sigma \cap B_\delta (z_0) = \{ x_n > G(x')\} \cap B_\delta (z_0),\]
for all $z_0 \in \partial\Sigma$ with $|z_0|=1$, where with $G$ convex.

But since the blow-up of $\Sigma$ (of $G$) is a plane for all $z\in B_\delta(z_0)$ we find that $G$ is $C^1$ and thus $\partial \Sigma$ is $C^1$ near $z_0$.
\end{proof}

We next show the following.

\begin{prop}\label{propclassif2}
Let $\Sigma$ be a closed convex cone in $\R^n$ with empty interior and vertex at $0$.
Then, Theorem \ref{thmclassif} holds for $\Omega=\Sigma$.
\end{prop}

We will need the following supersolution.

\begin{lem}\label{supersol-fully}
Let $s\in(0,1)$, and $L\in \mathcal L_*$.
Given $e\in S^{n-1}$, the function $\phi(x)=  \exp\bigl(-| e\cdot x|\bigr)$ is a viscosity supersolution of
\[ L\phi \le C \quad\textrm{in}\ \R^n\]
and satisfies the inequality pointwise.
The constant $C$ depends only on $n$, $s$, and ellipticity constants.
\end{lem}

\begin{proof}
The function $\phi_1(x) = \min\{4, \exp\bigl(- e\cdot x \bigr)\}$ satisfies
\[ L\phi_1 \le C \quad \mbox{in }\{e\cdot x > -1\}\]
and the function $\phi_2(x) = \min\{4, \exp\bigl(e\cdot x \bigr)\}$ satisfies
\[ L\phi_2 \le C \quad \mbox{in } \{e\cdot x < 1\}.\]
It immediately follows that $\phi = \min \{\phi_1,\phi_2\}$ is a viscosity supersolution in all of $\R^n$ and satisfies the inequality pointwise.
\end{proof}

We now give the:

\begin{proof}[Proof of Proposition \ref{propclassif2}]
Since $\Sigma$ is convex and has empty interior it will be contained on some hyperplane $\Sigma = \{e\cdot u = 0\}$.
Using Proposition \eqref{first-regularity-u} (rescaled) at every ball $B_R$ and \eqref{growthctrolgradient} we obtain
\[ [u]_{C^{1,\tau}(B_R)} \le CR^{s+\alpha-\tau}\]
In particular $u \in C^{1,\tau}_{\rm loc} (\R^n)$.

Let us show now that, given $h\in \R^n$, the function $v:= u-u(\cdot-h)$ is a viscosity subsolution of $Lv\ge 0$ in all of $\R^n$.
Indeed, for $\varepsilon \in (0,1)$ consider $v_\varepsilon = v- \varepsilon\phi$, where $\phi$ is the supersolution of Lemma \ref{supersol-fully}.
Note that $\phi$ has a positive wedge on $\Sigma = \{e\cdot u = 0\}$.

Assume $x_0\in \R^n$ and that $w$ is $C^2$ in a neighborhood of $x_0$ and touches $v_\varepsilon$ by above at $x_0$.
Since $v\in C^{1,\tau}_{\rm loc}(\R^n)$ then $v_\varepsilon$ has a negative wedge on $\Sigma$ and thus it can not be touched by above by a $C^2$ function on $\Sigma$. Thus $x_0$ does not belong to $\Sigma$. Then, we use that $Lv \ge 0$ in $\R^n\setminus \Sigma$ to obtain
\[  L v_\varepsilon(x_0) \ge Lv(x_0)  - \varepsilon L\varphi(x_0) \ge 0 - C\varepsilon\]
Then, $v = \sup_{\varepsilon>0} v_\varepsilon$ is a viscosity subsolution of $Lv \ge -C \varepsilon$
for all $\varepsilon>0$.
Therefore $Lv\ge0$ in the viscosity sense in all of $\R^n$, for every $h\in \R^n$, i.e.,
\[L \bigl( u-u(\,\cdot\, -h) \bigr) \ge 0.\]
Now, changing $h$ by $-h$ we get
$L\bigl( u(\,\cdot\, + h) -u \bigr)\le 0$, or equivalently $L\bigl( u -u(\,\cdot\, - h) \bigr)\le 0$ in all of $\R^n$.
It follows that, for any fixed $h$
\[ Lv = L \bigl( u-u(\,\cdot\, -h) \bigr) = 0 \quad \mbox{in all of }\R^n.\]
Using \eqref{growthctrolgradient} we have
\[ \|v\|_{L^\infty(B_R)}  \le |h| R^{s+\alpha} \]
for all $R\ge1$, and since $s+\alpha<2s$ the regularity theory for $Lv=0$ implies that $v$ is affine.
Thus, $u$ is a quadratic polynomial.
But since $u$ is convex, has a minimum at $0$, and it has subquadratic growth at $\infty$ (since $1+s+\alpha<2$) we obtain that $u \equiv 0$, as desired.
\end{proof}

We finally give the:

\begin{proof}[Proof of Theorem \ref{thmclassif}]
The proof is via a blow-down argument.
If $u\equiv 0$ there is nothing to prove.
Hence, we will assume that $u$ is not identically $0$ and thus $\Omega\neq\R^n$.

After a translation we may assume that $0\in \partial\Omega$.
Since $\Omega$ is convex, then we will then have
\[\Omega \subset \{e\cdot x\le 0\}\]
for some $e\in S^{n-1}$.

For $R\ge 1$ define
\[ \theta(R)  = \sup_{R'\ge R} \frac{ \|\nabla u\|_{L^\infty(B_{R'})} }{ (R')^{s+\alpha}}.\]
Note that $0< \theta(R)<\infty$ and that it is nonincreasing.

For all $m\in \mathbb N$  there  is $R'_m\ge m$ such that
\[ (R'_m)^{-s-\alpha} \|\nabla u_m\|_{L^\infty(B_{R_m})} \ge \frac{\theta(m)}{2} \ge \frac{\theta(R'_m)}{2}.\]
Then the blow down sequence
\[ u_m(x) := \frac{u(R'_m x)}{(R'_m)^{1+s+\alpha}\theta(R'_m)}\]
satisfies the growth control
\[ \|\nabla u_m\|_{L^\infty(B_{R})} \le R^{s+\alpha}\quad \mbox{for all }R\ge 1 \]
and the ``nondegeneracy'' condition
\[ \|\nabla u_m\|_{L^\infty(B_{1})}  \ge \frac 12. \]

By the $C^{1,\tau}$ estimates of Proposition \ref{first-regularity-u} and the Arzel\`a-Ascoli theorem, the functions $u_m$ converge (up to a subsequence) locally uniformly in $C^1$ to a function $u_\infty$ that
satisfies
\begin{equation}\label{growthuinfty}
\|\nabla u_\infty\|_{L^\infty(B_{R})} \le R^{s+\alpha}\quad \mbox{for all }R\ge 1,
\end{equation}
\begin{equation}\label{gnondeguinfty}
\|\nabla u_\infty\|_{L^\infty(B_{1})}  \ge \frac 12,
\end{equation}
and
\begin{equation}\begin{cases}
L (\nabla u_\infty) = 0  \quad &\mbox{in } \R^n\setminus \Sigma\\
L \bigl(u_\infty-u_\infty(\cdot-h)\bigr) \geq 0  \quad &\mbox{in } \R^n\setminus \Sigma\\
D^2 u_\infty \ge 0 \quad &\mbox{in } \R^n\\
u_\infty = 0 &\mbox{in }\Sigma,
\end{cases}\end{equation}
where
\begin{equation}\label{Sigma}
\Sigma = \bigcap_{k\geq1}  \Omega/R'_k \subset \{e\cdot x\le 0\}.
\end{equation}
Note that \eqref{Sigma} follows from the convexity of $\Omega$: since $0\in \partial\Omega$ then we have $\Omega\supset\Omega/R'_1\supset \Omega/R'_2\supset \dots\supset \Omega/R'_k$ for all $k\geq1$.

Notice that $\Sigma$ is a cone.

If $\Sigma$ has empty interior then it follows from Proposition \ref{propclassif2} that $u_\infty\equiv0$.
Hence, $\Sigma$ has nonempty interior.

Now, using the interior Harnack inequality and the fact that $\nabla u_\infty$ is not identically zero, we find that $\Sigma = \{\nabla u_\infty =0\} = \{u_\infty=0\}$, where the last identity is by convexity of $u_\infty$.
Thus, it follows from Proposition \ref{propclassif} that
\[ u_\infty(x) = C(e\cdot x)_+^{1+s} \quad \mbox{ and }\quad\Sigma = \{e\cdot x\le 0\}.\]
Using that $0\in \partial \Omega$, \eqref{Sigma}, and the convexity of $\Omega$, it then follows that
\[\Omega = \Sigma = \{e\cdot x\le 0\}.\]
Hence, using again Proposition \ref{propclassif} we find that
\[ u(x) = K(e\cdot x)_+^{1+s} \]
for some $K>0$, and the theorem is proved.
\end{proof}

\begin{rem}
In the second order case there are non-trivial global solutions with ellipsoids and paraboloids as zero sets.
These solutions have the same homogeneity at infinity (quadratic) as the half-space solution.
\end{rem}

\section{Regular points and blow-ups}
\label{sec5}

We start in this section the study of free boundary points.
More precisely, we show that at any regular point $x_0$ there is a blow-up of the solution $u$ that converges to a global convex solution of \eqref{eqnliouville} satisfying \eqref{growthctrolgradient}.

From now on, we consider the equivalent problem
\begin{equation}\label{obstacleproblem2}\begin{cases}
u \ge 0 \quad &\mbox{in }\R^n\\
Lu \le f  &\mbox{in }\R^n\\
Lu=f  &\mbox{in }\{u > 0\}\\
D^2 u \ge -1  &\mbox{in } \R^n.
\end{cases}\end{equation}
This is obtained by subtracting the obstacle $\varphi$ to the solution $u$ to \eqref{obstacle} and then dividing by $C\|\varphi\|_{C^{2,1}(\R^n)}$.
For convenience, we still denote $u$ the solution to \eqref{obstacleproblem2}.

Notice that, dividing by a bigger constant if necessary, we will have
\begin{equation}\label{f}
f\in C^1(\R^n) \quad \mbox{with}\quad \|f\|_{{\rm Lip}(\R^n)}\le 1.
\end{equation}
Moreover, by the results of Section \ref{sec2},
\begin{equation}\label{C-1-tau}
u\in C^{1,\tau}(\R^n) \quad \mbox{with}\quad \|u\|_{C^{1,\tau}(\R^n)}\le 1,
\end{equation}
for some $\tau>0$.

\begin{defi}\label{def-regular}
We say that a free boundary point $x_0 \in \partial \{u>0\}$ is \emph{regular} if
\begin{equation}\label{regular-point}
\limsup_{r\downarrow0}\frac{\sup_{B_r(x_0)}u}{r^{1+s+\alpha}} = \infty
\end{equation}
for some $\alpha\in(0,s)$ such that $1+s+\alpha<2$.

Notice that, according to this definition, non-regular points will be those at which $u(x)= O(|x-x_0|^{1+s+\alpha})$ for all such values $\alpha$.
\end{defi}

The definition of \emph{regular} free boundary point is qualitative.
In some of our results we need the following quantitative version.

\begin{defi}\label{def-regular-2}
Let $\nu:(0,\infty) \rightarrow (0,\infty)$ be a nonincreasing function with
\[\lim_{r\downarrow 0} \nu(r)= \infty.\]
We say that a free boundary point $x_0 \in \partial \{u>0\}$ is \emph{regular} with modulus $\nu$ if
\begin{equation}\label{regular-point}
\sup_{r'\ge r} \frac{\sup_{B_{r'}(x_0)}u}{(r')^{1+s+\alpha}} \ge \nu(r)
\end{equation}
for some $\alpha\in(0,s)$ such that $1+s+\alpha<2$.
\end{defi}

We next show the following result, which states that at any regular free boundary point $x_0$ there is a blow-up sequence that converges to $K(e\cdot x)_+^{1+s}$ for some $K\in\R$ and $e\in S^{n-1}$.

\begin{prop}\label{proprescalings}
Let $L\in\mathcal L_*$, and $u$ a be solution to \eqref{obstacleproblem2}-\eqref{f}-\eqref{C-1-tau}.
Assume that $0$ is a regular free boundary point with modulus $\nu$.

Then, given $\delta>0$, $R_0\ge 1$, and $r_0>0$ there is
\[ r = r(\delta, R_0, r_0, \alpha, \nu, s,n,\lambda,\Lambda) \in (0,r_0)\]
such that, for some $d>0$, the rescaled function
\[ v(x) : = \frac{u(r x)}{d} \]
satisfies, for some $e\in S^{n-1}$,
\begin{equation}\label{growthvk}
\|\nabla v\|_{L^\infty(B_{R})} \le R^{s+\alpha}\quad \mbox{for all }R\ge 1,
\end{equation}
\begin{equation}\label{approxsol}
\bigl| L (\nabla v)\bigr| \le \delta \quad  \mbox{in }\{v > 0\},\\
\end{equation}
and
\begin{equation}
\bigl| v - K(e\cdot x)_+^{1+s} \bigr|  +\bigl |\nabla v -  K(1+s) (e\cdot x)_+^{s}\,e\bigr| \le \delta \quad \mbox{in } B_{R_0}
\end{equation}
for some $K>0$ satisfying $\frac 14 \leq K\leq 1$.
\end{prop}

For this, we will need the following.

\begin{lem}\label{lema-theta}
Let $L\in\mathcal L_*$, and $u$ be the solution to \eqref{obstacleproblem2}-\eqref{f}-\eqref{C-1-tau}.
Assume that $x_0$ is a regular free boundary point with modulus $\nu$.
Then, the quantity
\begin{equation}\label{theta}
\theta(r) := \sup_{r'\ge r}  (r')^{-s-\alpha} \bigl\|\nabla u\bigr\|_{L^\infty(B_{r'}(x_0))}
\end{equation}
satisfies $\theta(r) \ge \nu(r)$ for all $r>0$.
\end{lem}

\begin{proof}
Since $u(x_0)=0$ for every $r'>0$ we have:
\[
(r')^{-1-s-\alpha} \|u\|_{L^\infty(B_{r'}(x_0))} \le (r')^{-s-\alpha}\bigl\|\nabla u\bigr\|_{L^\infty(B_{r'}(x_0))}
\]
and the result follows taking supremum in $r'\ge r$ and using the definition of $\nu$.
\end{proof}

To prove Proposition \ref{proprescalings} we will also need the following intermediate step.

\begin{lem}\label{lemconvergence}
Given $\delta>0$ and $R_0\ge 1$, there is
\[\eta=\eta(\delta,R_0, \alpha, n,s,\lambda,\Lambda)>0\]
such that the following statement holds:

Let $v\in {\rm Lip}(\R^n)$ be a nonnegative function satisfying
\[|L  (\nabla v\cdot e)|  \le \eta \quad \mbox{in }\{v>0\} \quad \mbox{ for all }e\in S^{n-1},\]
\[D^2 v \ge -\eta {\rm Id} \qquad\textrm{in}\quad \R^n,\]
\[\bigl\|\nabla v\bigr\|_{L^\infty(B_{R})}  \le R^{s+\alpha} \quad \mbox{for all }R\ge 1,\]
and
\[\bigl\|\nabla v\bigr\|_{L^\infty(B_{1})} \ge \frac{1}{2}.\]

Then, for some $e\in S^{n-1}$, we have
\[ \bigl| v - K(e\cdot x)_+^{1+s} \bigr|  +\bigl |\nabla v -  K(1+s) (e\cdot x)_+^{s}\,e\bigr| \le \delta \quad \mbox{in } B_{R_0}  \]
where
\[\frac 12 \le (1+s)K\le 1.\]
\end{lem}

\begin{proof}
The proof is by a compactness contradiction argument.
Assume that for some $R_0\ge 1$ and $\delta>0$ we have sequences $\eta_k\downarrow 0$, $L_k$ of the form \eqref{L*1}-\eqref{L*2},  and $v_k\in {\rm Lip}(\R^n)$ of functions satisfying
\[|L_k(\nabla v_k\cdot e)|  \le \eta_k \quad \mbox{in }\{v_k>0\} \quad \mbox{ for all }e\in S^{n-1},\]
\[D^2 v_k \ge -\eta_k {\rm Id} \qquad\textrm{in}\quad \R^n,\]
\[\bigl\|\nabla v_k\bigr\|_{L^\infty(B_{R})}  \le R^{s+\alpha} \quad \mbox{for all }R\ge 1,\]
and
\begin{equation}\label{nondegvk}
\bigl\|\nabla v_k\bigr\|_{L^\infty(B_{1})} \ge \frac{1}{2}
\end{equation}
but we have
\begin{equation}\label{contr}
 \bigl\| v_k - K(e\cdot x)_+^{1+s} \bigr\|_{C^1(B_{R_0})} \ge \delta \quad \mbox{for all } \frac 1 2\le(1+s)K\le 1\quad\textrm{and}\quad e\in S^{n-1}.
 \end{equation}

By Proposition \ref{first-regularity-u} (rescaled) we obtain that $v_k$ is $C^{1,\tau}$ in all of $\R^n$ with the estimate
\[ \bigl[\nabla v_k\bigr]_{C^\tau(B_{R})}  \le CR^{s+\alpha-\tau} \quad \mbox{for all }R\ge 1.\]
Thus, up to taking a subsequence, the operators $L_k$ converge weakly to some operator $L\in \mathcal L_*$ ---the spectral measures $a_k(y/|y|)$ converge weakly in $L^\infty(S^{n-1})$.
By Arzel\`a-Ascoli, the functions $v_k$ converge in $C^1_{\rm loc}(\R^n)$ to a function $v_\infty$, that is a viscosity solution of
\[L(\nabla v_\infty\cdot e) = 0 \quad \mbox{in }\{v_\infty>0\} \quad \mbox{ for all }e\in S^{n-1},\]
\[D^2 v_\infty  \ge 0 \qquad\textrm{in}\quad \R^n\]
with the growth control
\[\bigl\|\nabla v_\infty\bigr\|_{L^\infty(B_{R})}  \le R^{s+\alpha} \quad \mbox{for all }R\ge 1.\]

By the classification result Theorem \ref{thmclassif}, we have
\[ v_\infty \equiv K_0 (x\cdot e)_+^{1+s},\quad \mbox{for some } e\in S^{n-1} \mbox{ and } K_0\ge 0.\]
Passing \eqref{nondegvk} to the limit and using the growth control we have
\[ \frac 12 \le \bigl\|\nabla v_k\bigr\|_{L^\infty(B_{1})} \le 1\]
and thus  $\frac{1}{2} \le (1+s)K_0 \le 1$.

We have shown that $v_k\rightarrow K_0 (x\cdot e)_+^{1+s}$ in the $C^1$ norm, uniformly on compact sets.
In particular, \eqref{contr} is contradicted for large $k$, and thus the lemma is proved.
\end{proof}

\begin{proof}[Proof of Proposition \ref{proprescalings}]
We will deduce the result from Lemma \ref{lemconvergence}.
For this, we have to rescale the solution $u$ appropriately and check that the hypotheses of the lemma are satisfied.

Let
\[\theta(r)  :=   \sup_{r'\ge r}  (r')^{-s-\alpha} \bigl\|\nabla u\bigr\|_{L^\infty(B_{r'})}.\]
By assumption, $0$ is a regular point with modulus $\nu$.
Then, by Lemma \ref{lema-theta} we have
\[ \theta(r) \ge \nu(r)\rightarrow \infty\qquad \textrm{as}\ r\downarrow0.\]
Note that $\theta$ is nonincreasing.

Then, for every $m \in \mathbb N$ there is an
\[r'_{m}\ge \frac1m \]
such that
\begin{equation}\label{nondeg_proprescalings}
 (r'_{m})^{-s-\alpha} \bigl\|\nabla u\bigr\|_{L^\infty(B_{r'})} \ge \frac 12 \theta(1/m)\ge \frac 12 \theta(r'_m).
\end{equation}
Note that since $\|\nabla u\|_{L^\infty(\R^n)}\le 1$ we have
\[ (r'_{m})^{-s-\alpha}  \ge \frac 12 \theta(1/m)\ge \frac 12 \nu(1/m) \]
and thus
\[  \frac1m \le r'_m \le (\nu(1/m))^{-1/(s+\alpha)} \downarrow 0. \]
This shows that taking $m$ large enough we will have $r'_m\le r_0$.

Define the`` blow-up sequence''
\[v_{m}(x) := \frac{u(r'_{m} x)}{(r'_m)^{1+s+\alpha} \theta(r'_{m})}.\]
By definition of $\theta$, we have
\begin{equation}\label{growthcv_m1}
\bigl\|\nabla v_m\bigr\|_{L^\infty(B_{R})}  \le R^{s+\alpha} \quad \mbox{for all }R\ge 1,
\end{equation}
and by \eqref{nondeg_proprescalings}
\begin{equation}\label{nondegcv_m}
\bigl\|\nabla v_m\bigr\|_{L^\infty(B_{1})} \ge \frac{1}{2}.
\end{equation}

On the other hand, the function $v_m$ satisfies
\[\begin{split}
\bigl|L \nabla v_{m} \bigr|
&= \frac{(r'_{m})^{1+2s}}{(r'_{m})^{1+s+\alpha} \theta(r'_{m})}  \bigl|  L \nabla u (r'_{m} \,\cdot\,) \bigl|
\\
&\le  \frac{(r'_{m})^{1+2s}}{(r'_{m})^{1+s+\alpha} \theta(r'_{m})} \sup_{x\in \R^n} \bigl| \nabla f \bigl|
\\
&\le \frac{(r'_{m})^{s-\alpha}}{\theta(r'_m)}
\\
&\le \frac{1}{\nu(r'_{m})}  \le  \frac{1}{\nu(1/m)}  :=\eta_m
\end{split}\]
in the domain $\{v_{m}>0\}$.
Furthermore,
\[ D^2 v_{m} = \frac{(r'_{m})^{2}}{(r'_m)^{1+s+\alpha} \theta(r'_{m})} D^2 u
\ge -\frac{(r'_m)^{1-s-\alpha}}{\nu(1/m)}\,{\rm Id} \ge - \eta_m {\rm Id} \qquad\textrm{in}\quad \R^n.\]

Take now $m$ large enough (but fixed) so that $\eta_m\leq \eta$, where $\eta$ is the constant given by Lemma \ref{lemconvergence}.
Then, by Lemma \ref{lemconvergence} we obtain
\[\bigl| v_m - K(e\cdot x)_+^{1+s} \bigr|  +\bigl |\nabla v_m -  K(1+s) (e\cdot x)_+^{s}\,e\bigr| \le \delta \quad \mbox{in } B_{R_0}\]
with $1/2 \le (1+s)K \le 1$.
\end{proof}

\section{Optimal regularity of solutions and regularity of free boundaries}
\label{sec6}

We prove in this section Theorem \ref{th2}.
The proof will consist on several steps.

First, we prove that the free boundary will be Lipschitz near any regular point~$x_0$, with Lipschitz constant going to zero at that point.
Using this, we then prove that the set of regular points is open, and thus the free boundary is $C^1$ near those points.
Finally, we deduce that the free boundary will be $C^{1,\gamma}$, and that the solution will be $C^{1,s}$.

\subsection{Cones of monotonicity}

We prove first the following result, which states that the free boundary is Lipschitz in a neighborhood of any regular point $x_0$.
Moreover, the Lipschitz constant approaches zero as we approach $x_0$, so that the free boundary is $C^1$ at $x_0$.

\begin{prop}\label{free-bdry-Lipschitz}
Let $u$  be a solution of the obstacle problem \eqref{obstacleproblem2}-\eqref{f}-\eqref{C-1-tau} and assume that $x_0=0$ is a regular free boundary point with modulus $\nu$.
Then, there exists a vector $e\in S^{n-1}$ such that for any $\ell >0$ there is $r>0$ such that
\[ \{u>0\}\cap B_r  = \bigl\{\bar x_n > g(\bar x_1, \bar x_2,\dots, \bar x_{n-1}) \bigr\}\cap B_r \]
where $\bar x  = R x$, $R$ rotation with $Re = e_n$, and where $g$ is Lipschitz with
\[\|g\|_{C^{0,1}(B_r)}\le \ell.\]
Moreover,
\[\partial_{e'}u\geq0\quad\textrm{in}\ B_r,\quad\textrm{for all}\ e'\cdot e\geq\frac{\ell}{\sqrt {1+\ell^2}},\]
and
\[\partial_e u\geq cr^{s+\alpha}\quad\textrm{in}\ B_{r}(2re)\subset\{u>0\}.\]
The constants $c$ and $r$ depend only on $\ell$, $\alpha$, $\nu$, $n$, $s$, $\lambda$, $\Lambda$.
\end{prop}

For this, we will need the following:

\begin{lem}\label{lemdelta}
There is $\varepsilon=\varepsilon(n,s, \Lambda,\lambda)>0$ such that the following statement holds.

Let $E\subset B_1$ be some relatively closed set.
Assume that $w\in C(B_1)$ satisfies (in the viscosity sense)
\begin{equation}\label{eqnw}
\begin{cases}
Lw \le \varepsilon \quad &\mbox{in } B_1\setminus E\\
w = 0 &\mbox{in } E \cup (\R^n \setminus B_2)\\
w \ge -\varepsilon  &\mbox{in } B_2\setminus E,
\end{cases}
\end{equation}
and
\begin{equation}\label{wlargeatmanypoints}
\int_{B_1} w_+\,dx \ge 1.
\end{equation}
Then,
\[w\ge 0 \quad \mbox{in }\overline{ B_{1/2}}.\]
\end{lem}

\begin{proof}
Let $\psi \in C^\infty_c(B_{3/4})$ be some radial bump function with $\psi \ge 0$ and satisfying $\psi\equiv1$ in $B_{1/2}$.
Let
\[ \psi_t (x)  = - \varepsilon - t +  \varepsilon \psi(x).\]

If the conclusion of the lemma does not hold then $\psi_\varepsilon$ touches $w$ by below at $z\in  B_{3/4}$ for some $t>0$.
Since $\psi_t\le -t$ in all of $\R^n$ we have that $w(z) = \psi_t(z) <0$ and hence $z$ belongs to $B_1 \setminus E$.

By Lemma 3.3 in \cite{CS}, the operator $L$ can be evaluated classically at the point~$z$.

One the one hand we have
\[ L(w-\psi_t)(z) \ge \lambda\int_{\R^n} (w-\psi_t)(z+y) |y|^{-n-2s} \ge  \lambda\int_{B_1} w_+  \,dx \ge \lambda.\]
On the other hand
\[  L(w-\psi_t)(z)\le Lw(z) + |L \psi_t(z)| = \varepsilon + C\varepsilon.\]
We obtain a contradiction by taking $\varepsilon$ small enough.
\end{proof}

We now show Proposition \ref{free-bdry-Lipschitz}.

\begin{proof}[Proof of Proposition \ref{free-bdry-Lipschitz}]
Let $\delta>0$, $R_0\geq1$, and $r_0>0$ to be chosen later, and consider the rescaled function
\[ v(x)  = \frac{u(rx)}{d},\]
with $r\in(0,r_0)$ and $d>0$ given by Proposition \ref{proprescalings}.
Recall that for some $e\in S^{n-1}$ we have
\[\bigl|\nabla v(x)-K(1+s)(x\cdot e)_+^se\bigr| \leq \delta,\]
for some $\frac14\leq K\leq 1$.

Let $e'\in S^{n-1}$ with
\[e'\cdot e \ge \frac{\ell}{\sqrt {1+\ell^2}} \ge \frac \ell2\]
(we may assume that $\ell\le 1$).

Then, we have
\begin{equation} \label{largeatmanypoints}
\nabla v\cdot e' \geq \ell (e\cdot x)_+^{s} - \delta \quad \mbox{in }B_{R_0},
\end{equation}
and
\begin{equation} \label{eqnvk}
|L(\nabla v\cdot e')| \le \delta  \quad \mbox{in }\{v > 0\}.
\end{equation}

Let $\varepsilon>0$ be the universal constant from Lemma \ref{lemdelta}.
Then, for every given $\ell>0$ from \eqref{largeatmanypoints}, \eqref{eqnvk}, and the growth control \eqref{growthvk}, we see that we can choose $C$ universal large enough,
and $R_0$ large enough (depending only on  $\ell$, $\alpha$, $n$, $s$, $\delta$, $\lambda$ and $\Lambda$) so that the function
\[ w =  \frac{C_1}{\ell}  (\nabla v\cdot e') \chi_{B_2}\]
satisfies
\begin{equation}\label{eqnw1}
\begin{cases}
Lw \le  \frac{CC_1}{\ell}\,\delta\leq \varepsilon \quad &\mbox{in } B_1\setminus E\\
w = 0 &\mbox{in } E \cup (\R^n \setminus B_2)\\
w \ge -\frac{CC_1}{\ell}\,\delta\geq -\varepsilon  &\mbox{in } B_2\setminus E,
\end{cases}
\end{equation}
and
\[\int_{B_1} w_+\,dx \ge 1,\]
provided that $C_1$ is large enough and $\delta$ is small enough.

Then, Lemma \ref{lemdelta} implies that
\[w\ge 0 \quad \mbox{in }B_{1/2}.\]
Using that $v$ is a rescaling of $u$, this is equivalent to
\[ \partial_{e'} u \ge 0 \quad \mbox{in}\quad  B_{r/2}.\]
This happening for all $e'$ with $e'\cdot e \ge \frac{\ell}{\sqrt {1+\ell^2}}$ implies that $\{u>0\}$ is in $B_{r/2}$ a Lipchitz epigraph in de direction $e$ with Lipschitz constant bounded by $\ell$.

Finally, using \eqref{largeatmanypoints} we find
\[\partial_e u\geq cr^{s+\alpha}\quad\textrm{in}\ B_{r}(2re),\]
and the Proposition is proved.
\end{proof}

\subsection{$C^1$ regularity of free boundaries}

We prove now that the set of regular free boundary points is open, and that the free boundary is $C^1$ near those points.

The following lemma from \cite{RS-C1} states the existence of positive subsolutions of homogeneity $s+\epsilon$ vanishing outside of a convex cone that is very close to a half space.

\begin{lem}[\cite{RS-C1}]\label{homog-subsol}
Let $s\in(0,1)$, and $e\in S^{n-1}$.
For every $\epsilon>0$ there is $\eta>0$ such that the function
\[ \Phi(x) = \left( e\cdot x- \frac{\eta}{4} |x| \left(1- \frac{(e\cdot x)^2}{|x|^2} \right)\right)_+^{s+\epsilon}
\]
satisfies, for all $L\in \mathcal L_{*}$,
\[
\begin{cases}
L \Phi \ge 0 \quad & \mbox{in }\mathcal C_\eta  \\
\Phi = 0  \quad & \mbox{in }\R^n \setminus \mathcal C_\eta\\
\end{cases}
\]
where
\[\mathcal C_\eta: = \left\{ x \in \R^n\ : \ e\cdot\frac{x}{|x|} >  \frac{\eta}{4} |x| \left(1- \frac{(e\cdot x)^2}{|x|^2} \right) \right\}.\]
The constant $\eta$ depends only on $\epsilon$, $s$, and ellipticity constants.
\end{lem}

Using the previous Lemma, we now show the following.

\begin{prop}\label{contagi}
Let $u$  be a solution of the obstacle problem \eqref{obstacleproblem2}-\eqref{f}-\eqref{C-1-tau} and assume that $0$ is a regular free boundary point with modulus $\nu$.
Then, there is $r>0$ such that every point on $\partial \{u>0\}\cap B_r$ is regular, with a common modulus of continuity $\tilde\nu$.
\end{prop}

\begin{proof}
By Proposition \ref{free-bdry-Lipschitz}, there is $e\in S^{n-1}$ such that given $\eta>0$ small  there is $r>0$ for which
\[(x_0 + \mathcal C_\eta)\cap B_{2r}(x_0) \subset \{u>0\} \cap B_{4r}(0) \quad \mbox{for all }x_0 \in  \{u>0\} \cap B_r(0).\]
Here, $\mathcal C_\eta$ is the cone in Lemma \ref{homog-subsol}.

Hence, the rescaled function $\tilde u(x) = u(rx)$ is a solution of the obstacle problem satisfying
\[ (x_0 + \mathcal C_\eta)\cap B_2(\tilde x_0) \subset \{\tilde u>0\} \cap B_{4} \quad \mbox{for all }x_0 \in \{\tilde u>0\} \cap B_{1/4} \]
and,
\[ \bigl|L(\partial_e \tilde u)\bigr| \le C_1 r^{2s} \quad \mbox{ in }\{u>0\} \cap B_4.\]

By Proposition \ref{free-bdry-Lipschitz} we have
\[ \partial_e \tilde u \ge c_2 r^{s+\alpha} >0 \quad\mbox{ in }\overline{B_{1}(2e)} .\]

But from the homogeneous solution $\Phi$ in Lemma \ref{homog-subsol}, which has homogeneity $s+\epsilon$, we build the subsolution
\[ \psi = \Phi \chi_{B_4} + C_3\chi_{B_{1/4}(2e)}.\]
Indeed, if $C_3>1$ is large enough, then $\psi$ satisfies
\[ L\psi \ge 1 \quad \mbox{in } B_1\setminus B_{1/4}(2e) \]
and
\[\psi =0 \quad \mbox{outside }\mathcal C_\eta.\]

We now use  the translated function  $c_2 r^{s+\alpha} \psi(x- \tilde x_0) /C_3$, with $\tilde x_0\in \{\tilde u>0\} \cap B_{1/4}$ as lower barrier.
Taking $r$ small so that $C_1 r^{2s}> c_2 r^{s+\alpha} /C_3$, we will have $\partial_e \tilde u(x) \ge c_2 r^{s+\alpha} \psi(\tilde x_0) /C_3$.
Thus, by the maximum principle,
\[ \partial_e \tilde u \ge c_2 r^{s+\alpha} \psi(x-\tilde x_0) /C_3,\]
and using that $\Phi$ is homogeneous with exponent $s+\epsilon$, we find
\[ \partial_e \tilde u (\tilde x_0+te)  > c t^{s+\epsilon} \]
for $t\in (0,1)$.
In particular, $\|\nabla u\|_{B_t(x_0)}\geq ct^{s+\epsilon}$, and therefore $x_0$ is a regular point (with $\tilde\nu(t)=ct^{\epsilon-\alpha}$).
\end{proof}

Using the previous result, we find the following.

\begin{prop}\label{prop6.5}
Let $u$ be a solution of the obstacle problem \eqref{obstacleproblem2}-\eqref{f}-\eqref{C-1-tau} and assume $0$ is a regular point.
Then, there is $r>0$ such that $\partial \{u>0\}\cap B_r$ is a $C^1$ graph.
\end{prop}

\begin{proof}
By Propositions \ref{free-bdry-Lipschitz} and \ref{contagi} there is $r>0$ such that $\Gamma := \partial\{u>0\}\cap B_r$ is a Lipschitz graph (in some direction $e$) and every point in $\Gamma$ is a regular point.
Moreover, by Proposition \ref{free-bdry-Lipschitz} we have $\partial_e u \ge 0$ in $B_r$ and $\partial_e u$ is not identically $0$ in $B_r$ ---otherwise $\Gamma$ would not be contained in the free boundary.
Furthermore, by Proposition \ref{contagi}, all points in $\Gamma$ are regular points with a common modulus of continuity.

Thus, applying again Proposition \ref{free-bdry-Lipschitz} ---now  at every regular point $x_0\in \Gamma$ and with $\ell\searrow 0$--- we find that $\{u>0\}$ is $C^1$ at every point $x_0\in \Gamma$, with a uniform modulus of continuity.
\end{proof}

\subsection{Proof of Theorem \ref{th2}}

We now prove the optimal $C^{1,s}$ regularity of solutions and the $C^{1,\gamma}$ regularity of free boundaries.
For this, we will need the following from \cite{RS-C1}.

\begin{thm}[\cite{RS-C1}]\label{thmC1}
Let $s\in(0,1)$ and $\gamma\in(0,s)$.
Let $L\in\mathcal L_*$, and $\Omega$ be any $C^1$ domain.

Then, there exists is $\delta>0$, depending only on $\gamma$, $n$, $s$, $\Omega$, and ellipticity constants, such that the following statement holds.

Let $v_1$ and $v_2$, be viscosity solutions of
\begin{equation}\label{pb-C1}
\left\{ \begin{array}{rcll}
L v_i &=&g_i&\textrm{in }B_1\cap \Omega \\
v_i&=&0&\textrm{in }B_1\setminus\Omega,
\end{array}\right.\end{equation}
Assume that $\|g_i\|_{L^\infty(B_1\cap \Omega)}\leq C_0$, $\|v_i\|_{L^\infty(\R^n)}\leq C_0$,
\[g_i\geq-\delta\quad\textrm{in}\quad B_1\cap \Omega,\]
and that
\[v_i\geq0\quad\mbox{in}\quad \R^n,\qquad \sup_{B_1}v_i\geq 1.\]
Then,
\[ \|v_1/v_2\|_{C^\gamma(\Omega\cap B_{1/2})} \le CC_0,\qquad \gamma\in(0,s),\]
where $C$ depends only on $\gamma$, $n$, $s$, $\Omega$, and ellipticity constants.
\end{thm}

We will also need the following result.

\begin{thm}[\cite{RS-C1}]\label{thmC1alpha}
Let $s\in(0,1)$ and $\gamma\in(0,s)$.
Let $L\in\mathcal L_*$, $\Omega$ be any $C^{1,\gamma}$ domain, and $d$ be the distance to $\partial\Omega$.
Let $v$ be any solution to
\begin{equation}\label{pb-C-1-gamma}
\left\{ \begin{array}{rcll}
L v &=&g&\textrm{in }B_1\cap \Omega \\
v&=&0&\textrm{in }B_1\setminus\Omega,
\end{array}\right.\end{equation}
Then,
\[\|v/d^s\|_{C^{\gamma}(B_{1/2}\cap\overline\Omega)}\leq C\left(\|g\|_{L^\infty(B_1\cap\Omega)}+\|v\|_{L^\infty(\R^n)}\right).\]
The constant $C$ depends only on $n$, $s$, $\Omega$, and ellipticity constants.
\end{thm}

We now give the:

\begin{proof}[Proof of Theorem \ref{th2}]
After subtracting the obstacle $\varphi$ and dividing by a constant, $u$ satisfies \eqref{obstacleproblem2}-\eqref{f}-\eqref{C-1-tau}.

According to our definition of regular points, if $x_0$ is a free boundary point which is not regular then (ii) holds for all $\alpha\in(0,s)$ satisfying $1+s+\alpha<2$.
Note that after subtracting the obstacle we have $\varphi\equiv0$.

We next show that (i) holds at every regular point $x_0$, and that the set of regular points is relatively open and the free boundary is $C^{1,\gamma}$ for all $\gamma\in(0,s)$ near each regular point.

First, by Proposition \ref{prop6.5} the set of regular points is relatively open and the free boundary is a $C^1$ near these regular points. Let $x_0$ a regular point. After a rotation we may assume that $e_n$ is the unit inwards normal to $\{u>0\}$ at $x_0$.

\vspace{2mm}

{\em Step 1.} Let us prove that the free boundary is $C^{1,\gamma}$ near $x_0$.
Let $v_2 = \partial_n u$ and $v_1 = 2\partial_n u +\partial_i u$, where $1\le i\le n-1$.
We will use Proposition \ref{thmC1} to show that, for some $r>0$ we have
\begin{equation}\label{quocCalpha}
\left\| \frac{v_1}{v_2}\right\|_{C^{\gamma}(\{u>0\}\cap B_r(x_0))} =
\left\|2 + \frac{\partial_i u}{\partial_n u}\right\|_{C^{\gamma}(\{u>0\}\cap B_r(x_0))}  \le C.
\end{equation}
This will imply that the normal vector $\nu(x)$ to the level set $\{u=t\}$ for $t>0$ and $u(x)=t$, which is given by,
\[ \nu^i (x)  = \frac{\partial_i u }{|\nabla u|} (x) = \frac{\partial_i u/ \partial_n u}{ \sqrt{\sum_{j=0}^{n-1} (\partial_j u/ \partial_n u)^2  +1} }  \]
\[ \nu^n (x)  = \frac{\partial_n u }{|\nabla u|} (x) = \frac{1}{ \sqrt{\sum_{j=0}^{n-1} (\partial_j u/ \partial_n u)^2  +1} }  \]
satisfies $|\nu(x)-\nu(\bar x)|\le C |x-\bar x|^{\gamma}$ whenever $x,\bar x\in \{u=t\}\cap B_r(x_0)$, with $C$ independent of $t>0$.
Therefore, letting $t\downarrow 0$, we will find that $\partial \{u>0\}\cap \cap B_r(x_0)$ is a $C^{1,\gamma}$ graph, and Step 1 will be completed.

It remains to show \eqref{quocCalpha}.
To prove it, notice that we have $Lv_1=g_1$ and $Lv_2=g_2$ in $\Omega=\{u>0\}$, and $v_1=v_2=0$ in $\Omega^c$, with $|g_i|\leq C$.
Moreover, by Proposition \ref{free-bdry-Lipschitz},  $v_i\geq0$ in $B_r(x_0)$ and $\sup_{B_r(x_0)} v_i\geq cr^s$.
Thus, the rescaled functions $w_1(x):=Cr^{-s}v_1(x_0+rx)\chi_{B_{r_0}}(rx)$ and $w_2(x)=Cr^{-s}v_2(x_0+rx)\chi_{B_{r_0}}(rx)$, for $r>0$ small enough, satisfy the hypotheses of Theorem \ref{thmC1}.
Therefore, we get that
\[\left\| \frac{w_1}{w_2}\right\|_{C^{\gamma}(\Omega\cap B_{r_0/2})}\le C,\]
and thus \eqref{quocCalpha} follows.

\vspace{2mm}

{Step 2.} By Step 1, the domain $\{u>0\}\cap B_r(x_0)$ is $C^{1,\gamma}$.
Thus, we can apply Theorem \ref{thmC1alpha} to the partial derivatives $\partial_i u$, where $1\le i\le n$, to obtain that
\[ \|\partial_i u/d^s\|_{C^\gamma(\Omega\cap B_{1/2})}\le CC_0.\]
This implies that
\[ u(x) = c d^{1+s}(x) + o(|x-x_0|^{1+s+\gamma}).\]
Since $u\geq0$ then $c\geq0$, and since by $u$ does not satisfy (ii) then $c>0$.
\end{proof}

\section{Fully nonlinear equations: Classification of global solutions}
\label{sec-fully-1}

In this section we classify global convex solutions to the obstacle problem for fully nonlinear operators.

Throughout the section, we denote
\[M^+u:=M^+_{\mathcal L_*}u=\sup_{L\in \mathcal L_*}Lu,\]
and $M^-=M^-_{\mathcal L_*}$.
Moreover, $\bar\alpha>0$ is the constant given by Definition \ref{defi-baralpha}.
Recall that $\bar \alpha$ stays positive as $s\to 1$.

We establish two classification results.
The first one is the following, and will correspond to the case (i) in Theorem \ref{th3}.

\begin{thm}\label{thmclassif-fully}
Let $\gamma\in(0,1)$ and $\Omega\subset \R^n$ be a closed convex set which is {\em not} contained in the strip $\{-C\le e'\cdot x\le C\}$ for any $e'\in S^{n-1}$ and any $C>0$.

Let $\alpha\in(0,\bar\alpha)$ be such that $1+s+\alpha<2$.
Let $u\in C^{1}(\R^n)$ be a function satisfying
\begin{equation}\label{eqnliouville-fully}\begin{cases}
M^+(\partial_e u ) \geq 0 \ge M^-(\partial_e u )   \quad &\mbox{in } \R^n\setminus\Omega\quad \mbox{for all } e\in S^{n-1}.\\
u=0 & \mbox{in }\Omega\\
D^2 u \ge 0 \quad &\mbox{in } \R^n\\
u\ge 0 & \mbox{in } \R^n.
\end{cases}\end{equation}
Assume that $u$ satisfies the following growth control
\begin{equation}\label{growthctrolgradient-fully}
\|\nabla u\|_{L^\infty(B_R)} \le R^{s+\alpha} \quad \mbox{ for all }R\ge 1.
\end{equation}
Then, either $u \equiv 0$, or
\[ \Omega =\{e\cdot x\le 0\}\quad \mbox{ and }\quad u(x) = K (e\cdot x)_+^{1+s} \]
for some $e\in S^{n-1}$ and $K>0$.
\end{thm}

The second classification result, stated next, corresponds to case (ii) in Theorem~\ref{th3}.

\begin{thm}\label{thmclassif-fully2}
Let $\Omega\subset \R^n$ be a closed convex set which is contained in the strip $\{-C\le e'\cdot x\le C  \}$ for some $e'\in S^{n-1}$ and $C>0$.

Let $\beta\in(0,\bar\alpha)$ be such that $2s+\beta<2$.
Let $u\in  {\rm Lip}_{\rm loc}(\R^n)$ be a function satisfying
\begin{equation}\label{eqnliouville-fully21}
M^+\left( u - \ave u(\cdot-h)\,d\mu(h) \right) \geq 0 \quad \mbox{in }\R^n \setminus \Omega
\end{equation}
for every measure $\mu\ge0$ with compact support and $\mu(\R^n)=1$.
Assume in addition that
\begin{equation}\label{eqnliouville-fully22}
\begin{cases}
D^2 u \ge 0 \quad &\mbox{in } \R^n\\
u=0 & \mbox{in }\Omega\\
u\ge 0 & \mbox{in } \R^n,
\end{cases}\end{equation}
and that $u$ satisfies the growth control
\begin{equation}\label{growthctrol-fully}
\|\nabla u\|_{L^{\infty}(B_R)} \le R^{2s+\beta-1} \quad   \mbox{ for all }R\ge 1.
\end{equation}
Then, $u \equiv 0$.
\end{thm}

We next prove Theorems \ref{thmclassif-fully} and \ref{thmclassif-fully2}.

\subsection{Cones with nonempty interior}

We prove here Theorem \ref{thmclassif-fully}.
To prove it, we need the following.

\begin{prop}\label{thm-uniqueness-fully}
Let $\Sigma\subset \R^n$ be any convex cone with nonempty interior, with vertex at $0$, and such that $\partial \Sigma$ is $C^1$ away from $0$.
Let $v_1,v_2$ be functions in $C(\R^n)$ satisfying
\[\int_{\R^n} v_i(y) (1+|y|)^{-n-2s}\,dy  < \infty.\]
Assume that $v_1,v_2$ satisfy, for each $A,B\in\R$,
\begin{equation}\label{eqn1space}\begin{cases}
M^+(av_1+bv_2)  \geq 0,\quad M^-(av_1+bv_2)\leq0  \quad &\mbox{in } \R^n\setminus \Sigma\\
v_i=  0 & \mbox{in }\Sigma\\
v_i> 0 & \mbox{in } \R^n \setminus \Sigma.
\end{cases}\end{equation}
Then,
\[ v_1 \equiv K v_2 \quad \mbox{ in } \R^n\]
for some $K>0$.
\end{prop}

\begin{proof}
The proof is exactly the same as the one of Proposition \ref{thm-uniqueness}.
The only difference is that here we need to use a boundary Harnack principle in $C^1$ domains for equations with bounded measurable coefficients, given by Theorem 1.6 in \cite{RS-C1}.
\end{proof}

We give now the:

\begin{proof}[Proof of Theorem \ref{thmclassif-fully}]
The proof is essentially the same as the one given in Section~\ref{sec4}, using Proposition \ref{thm-uniqueness-fully} instead of Theorem \ref{thm-uniqueness}.
First, by a blow-down argument we only need to show the case in which $\Omega$ is a cone $\Sigma$.
Since $\Omega$ is not contained in any strip $\{|e'\cdot x|\leq C\}$ then the cone $\Sigma$ has nonempty interior.

Then, by a dimension-reduction argument we get that the cone $\Sigma$ is $C^1$ outside the origin.
This is done by a blow-up argument on lateral points of the cone.

Finally, since the cone $\Sigma$ is $C^1$, then any two derivatives $v_1=\partial_{e_1}u$ and $v_2=\partial_{e_2}u$, with $-e_i\in \Sigma$, satisfy \eqref{eqn1space}.
By Proposition \ref{thm-uniqueness-fully} we get that all such derivatives are equal up to multiplicative constant, and this yields that $\Sigma=\{x\cdot e\leq 0\}$ for some $e\in S^{n-1}$.
Finally, by the classification result \cite[Proposition 5.1]{RS-K} we get $u(x)=c(x\cdot e)_+^{1+s}$, and the Theorem is proved.
\end{proof}

\subsection{Cones with empty interior}

We next state the Liouville theorem that serves to prove $C^{2s+\gamma}$ interior regularity of concave fully nonlinear equations.
We will use it for convex equations but we state it for concave to obtain a more easily comparable statement and proof to that of Theorem 2.1 in \cite{S-convex}.
Also, we denote
\[2s=\sigma\]

\begin{prop}\label{liouville}
Let $\sigma_0\in(0,2)$ and $\sigma\in [\sigma_0,2)$.
There is $\bar \alpha>0$ depending only on $n$, $\sigma_0$, and ellipticity constants such that the following statement holds.

Let $\gamma\in(0,\bar\alpha)$.
Assume that $u\in {\rm Lip}_{\rm loc}(\R^n)$ satisfies the following properties.
\begin{enumerate}
\item[(i)] There exists $C_1>0$ such that for all for all $R\ge1$ we have
\[  [u]_{{\rm Lip}(B_R)}\le C_1 R^{\sigma+ \gamma-1}\]

\item[(ii)] We have
\[  D^2 u \ge 0 \quad \mbox{and} \quad u\ge 0 .\]

\item[(iii)] For every nonnegative measure $\mu$ in $\R^n$ with compact support and  $\int_{\R^n} \mu(h)\, dh =1$ we have, in the viscosity sense,
\[ M^+_{\mathcal L_0}  \left(\ave u(\,\cdot\,+h) \mu(h)\,dh- u\right)\ge 0\quad \mbox{in }\R^n.\]
\end{enumerate}
Then, either $u\equiv 0$  or $\sigma+\gamma>2$ and $u$ is a quadratic polynomial.
\end{prop}

In the case that  $u\in C^{2s+\epsilon}$ for some $\epsilon>0$ this follows from Theorem 2.1 in \cite{S-convex}.
Here, we give a variation that applies to any convex solution (that is, only Lipschitz a priori).
The proof of this result is differed to the Appendix.

We give here the:

\begin{proof}[Proof of Theorem \ref{thmclassif-fully2}]
By the exact same blow-down argument from Section \ref{sec4}, we may assume that $\Omega$ is a cone $\Sigma$.
Moreover, since $\Omega$ is contained in a strip $\{|e'\cdot x|\leq C\}$ then $\Sigma$ has zero measure.

Given a measure $\mu$ with compact support and unit mass we consider
\[v =  u- \ave u(\,\cdot\,+h) \mu(h)\,dh,\]
which satisfies $M^+ v \ge 0$ in $\R^n\setminus \Sigma$ is the viscosity sense.
Since $u$ is locally $C^{1,\tau}$ ---by Proposition \ref{first-regularity-u-fully}--- and since $\mu$ has compact support and unit mass, then $v$ is also locally $C^{1,\tau}$. Then using the supersolution $\phi$ of Lemma \ref{supersol-fully}, which satisfies
\[M^+\phi \le C \quad \mbox{in all of }\R^n\]
we obtain that
\[ M^+(v- \varepsilon \phi) \ge -C\varepsilon\quad \mbox{in all of }\R^n\]
for all $\varepsilon >0$.

Thus, taking $\varepsilon\downarrow 0$, we find that $u$ satisfies all the assumptions of Proposition \ref{liouville}.
It follows that $u$ is a quadratic polynomial. Since $u\ge0$, $u= 0$ on $\Sigma$, and $u$ has subquadratic growth by assumption, it must be $u\equiv 0$.
\end{proof}

\section{Fully nonlinear equations: Regularity of solutions and free boundaries}
\label{sec-fully-2}

Using Theorems \ref{thmclassif-fully} and \ref{thmclassif-fully2}, we now study the regularity of solutions and free boundaries.

Recall that the solution $u$ to \eqref{obst-pb-fully} can be constructed as the smallest supersolution lying above the obstacle and being nonnegative at infinity.
Thus, exactly as in Section \ref{sec2} we find the following.

\begin{lem}\label{semiconvex-fully}
Let $I$ be any operator of the form \eqref{I}, $\varphi$ be any obstacle satisfying \eqref{obstacle}, and $u$ be the solution to \eqref{obst-pb-fully}.
Then,
\begin{itemize}
\item[(a)] $u$ is semiconvex, with \[\hspace{40mm}\partial_{ee}u\ge -C\qquad \textrm{for all}\ e\in S^{n-1}.\]
\item[(b)] $u$ is bounded, with \[\|u\|_{L^\infty(\R^n)}\leq C.\]
\item[(c)] $u$ is Lipschitz, with \[\|u\|_{\rm Lip(\R^n)}\leq C.\]
\end{itemize}
The constants $C$ depend only on $\|\varphi\|_{C^{1,1}(\R^n)}$.
\end{lem}

\begin{proof}
See the proof of Lemma \ref{semiconvex}.
\end{proof}

Now, notice that after subtracting the obstacle $\varphi$ we get the following equation:
\[ 0= I(u-\varphi+\varphi) = \sup_{a\in \mathcal A} \bigl(L_a (u-\varphi) + c_a + L_a\varphi\bigr)\]
Thus, we consider the following problem which is equivalent to the obstacle problem with convex fully nonlinear elliptic operator:
\begin{equation}\label{obstacleproblem2fully}\begin{cases}
u \ge 0 \quad &\mbox{in }\R^n\\
\sup_{a\in \mathcal A}\bigl(L_a u + c_a(x) \bigr) \le 0 &\mbox{in }\R^n\\
\sup_{a\in \mathcal A}\bigl(L_a u + c_a(x) \bigr) =0  &\mbox{in }\{u > 0\}\\
D^2 u \ge -{\rm Id}  &\mbox{in } \R^n.
\end{cases}\end{equation}
This is obtained by subtracting the obstacle $\varphi$ to the solution $u$ to \eqref{obstacle} and then dividing by $C\|\varphi\|_{C^{2,1}(\R^n)}$.
For convenience, we still denote $u$ the solution to \eqref{obstacleproblem2fully}.

Notice that, dividing by a bigger constant if necessary, we will have
\begin{equation}\label{ffully}
c_a \in C^1(\R^n) \quad \mbox{with}\quad \|c_a\|_{{\rm Lip}(\R^n)}\le 1.
\end{equation}
Moreover, by Proposition \ref{first-regularity-u-fully}, we will have
\begin{equation}\label{C-1-taufully}
u\in C^{1,\tau}(\R^n) \quad \mbox{with}\quad \|u\|_{C^{1,\tau}(\R^n)}\le 1,
\end{equation}
for some $\tau>0$.

\subsection{Regular points and blow-ups}

From now on, we assume that $u$ is a solution of \eqref{obstacleproblem2fully}-\eqref{ffully}-\eqref{C-1-taufully}, and that $x_0$ is a regular free boundary point with exponent $\alpha$ and modulus $\nu$ in the sense of Definition \ref{def-regular-2}.

In case that $1+s+\alpha\geq 2s+\bar\alpha$, we will assume in addition that
\begin{equation}\label{non-singular}
\liminf_{r\to0}\frac{|\{u=0\}\cap B_r(x_0)|}{|B_r(x_0)|}>0.
\end{equation}
For such free boundary points we have the following.

\begin{prop}\label{proprescalingsfully}
Let $u$ a be solution to \eqref{obstacleproblem2fully}-\eqref{ffully}-\eqref{C-1-taufully}.
Assume that $x_0=0$ is a regular free boundary point with exponent $\alpha$ and modulus $\nu$.
In case $1+s+\alpha\geq2s+\bar\alpha$, assume in addition that \eqref{non-singular} holds.

Then, given $\delta>0$, $R_0\ge 1$, $r_0>0$ there is
\[ r=r(\delta, R_0, r_0, \alpha, \nu, s,n,\lambda,\Lambda) \in (0, r_0)\]
such that, for some $d>0$, the rescaled function
\[ v(x) : = \frac{u(r x)}{d} \]
satisfies
\begin{equation}\label{growthvkfully}
\|\nabla v\|_{L^\infty(B_{R})} \le R^{s+\alpha}\quad \mbox{for all }R\ge 1,
\end{equation}
\begin{equation}\label{approxsol fully}
M^+ (\partial_e v) \ge -\delta \quad  \mbox{in }\{v > 0\}\quad\mbox {for all }e\in S^{n-1},
\end{equation}
and
\[M^+\left( v - \ave v(\cdot-h)\,d\mu(h) \right) \geq -\delta {\rm diam}\bigl( {\rm spt} \,\mu\bigr) \quad \mbox{in }\{v>0\}\]
for every measure $\mu\ge0$ with compact support and $\mu(\R^n)=1$.
Moreover,
\begin{equation}
\bigl| v - K(e\cdot x)_+^{1+s} \bigr|  +\bigl |\nabla v -  K(1+s) (e\cdot x)_+^{s}\,e\bigr| \le \delta \quad \mbox{in } B_{R_0}
\end{equation}
for some $K>0$ satisfying $\frac 14 \leq K\leq 1$ and some $e\in S^{n-1}$.
\end{prop}

To prove Proposition \ref{proprescalings} we will also need the following intermediate step, which is the analogue of Lemma \ref{lemconvergence}.

\begin{lem}\label{lemconvergencefully}
Given $R_0\ge 1$ and $\delta> 0$, there is
\[\eta=\eta(\delta,R_0, n,s,\lambda,\Lambda)>0\]
such that the following statement holds.

Let $v\in {\rm Lip}(\R^n)$ be a nonnegative function satisfying
\[M^+ (\partial_e v) \ge -\eta \quad  \mbox{in }\{v > 0\}\]
for all $e\in S^{n-1}$,
\[ M^+\left( v - \ave v(\cdot-h)\,d\mu(h) \right) \geq -\eta {\rm diam}\bigl( {\rm spt} \,\mu\bigr) \quad \mbox{in }\{v>0\}\]
for every measure $\mu\ge0$ with compact support and $\mu(\R^n)=1$,
\[D^2 v \ge -\eta {\rm Id} \qquad\textrm{in}\quad \R^n,\]
\[\bigl\|\nabla v\bigr\|_{L^\infty(B_{R})}  \le R^{s+\alpha} \quad \mbox{for all }R\ge 1,\]
and
\[\bigl\|\nabla v\bigr\|_{L^\infty(B_{1})} \ge \frac{1}{2}.\]
In case $1+s+\alpha\geq 2s+\bar\alpha$, assume in addition
\[\frac{|\{v=0\}\cap B_R|}{|B_R|}\geq c_0>0\quad\mbox{for all}\ R\leq \frac{1}{\eta}.\]

Then,
\[ \bigl| v - K(e\cdot x)_+^{1+s} \bigr|  +\bigl |\nabla v -  K(1+s) (e\cdot x)_+^{s}\,e\bigr| \le \delta \quad \mbox{in } B_{R_0}  \]
where
\[\frac 12 \le (1+s)K\le 1\]
and $e\in S^{n-1}$.
\end{lem}

\begin{proof}
The proof is by a compactness contradiction argument and is very similar to that of Lemma \ref{lemconvergence}.

Assume that for some $R_0\ge 1$ and $\delta>0$ we have sequences $\eta_k\downarrow 0$ and $v_k\in {\rm Lip}(\R^n)$ satisfying
\[M^+ (\partial_e v_k) \ge -\eta_k \quad  \mbox{in }\{v > 0\}\mbox {for all }e\in S^{n-1} \quad \mbox{ for all }e\in S^{n-1},\]
\[  M^+\left( v_k - \ave v_k(\cdot-h)\,d\mu(h) \right) \ge -\eta_k  {\rm diam}\bigl( {\rm spt} \,\mu\bigr), \]
\[D^2 v_k \ge -\eta_k {\rm Id} \qquad\textrm{in}\quad \R^n,\]
\[\bigl\|\nabla v_k\bigr\|_{L^\infty(B_{R})}  \le R^{s+\alpha} \quad \mbox{for all }R\ge 1,\]
\begin{equation}\label{nondegvk8}
\bigl\|\nabla v_k\bigr\|_{L^\infty(B_{1})} \ge \frac{1}{2},
\end{equation}
and
\[\frac{|\{v_k=0\}\cap B_R|}{|B_R|}\geq c_0>0\quad\mbox{for all}\ R\leq \frac{1}{\eta_k}.\]
but we have
\begin{equation}\label{contr8}
 \bigl\| v_k - K(e\cdot x)_+^{1+s} \bigr\|_{C^1(B_{R_0})} \ge \delta \quad \mbox{for all } \frac 1 2\le(1+s)K\le 1\quad\textrm{and}\quad e\in S^{n-1}.
\end{equation}

By Proposition \ref{first-regularity-u-fully} we obtain that $v_k$ is $C^{1,\tau}$ in all of $\R^n$ with the estimate
\[ \bigl[\nabla v_k\bigr]_{C^\tau(B_{R})}  \le CR^{s+\alpha-\tau} \quad \mbox{for all }R\ge 1.\]
Thus, up taking subsequences, the functions $v_k$ converge in $C^1_{\rm loc}(\R^n)$ to a function $v_\infty$, that is a viscosity
\[M^+(\partial_e v_\infty) \ge 0 \quad \mbox{in }\{v_\infty>0\} \quad \mbox{ for all }e\in S^{n-1},\]
\[ M^+\left( v_\infty- \ave v_\infty(\cdot-h)\,d\mu(h) \right) \ge 0,\]
\[D^2 v_\infty  \ge 0 \qquad\textrm{in}\quad \R^n\]
with the growth control
\[\bigl\|\nabla v_\infty\bigr\|_{L^\infty(B_{R})}  \le R^{s+\alpha} \quad \mbox{for all }R\ge 1.\]
Moreover, if $2s+\bar\alpha\leq 1+s+\alpha$ then we also have
\[\frac{|\{v_\infty=0\}\cap B_R(x_0)|}{|B_R(x_0)|}\geq c_0>0\quad\mbox{for all}\ R.\]
In particular, the convex set $\{v_\infty=0\}$ is not contained in any strip $\{|e'\cdot x|\leq C\}$.

Thus, in case $2s+\bar\alpha\leq 1+s+\alpha$ by Theorem \ref{thmclassif-fully} we find
\[ v_\infty \equiv K_0 (x\cdot e)_+^{1+s},\quad \mbox{for some } e\in S^{n-1} \mbox{ and } K_0\ge 0.\]
In case $2s+\bar\alpha>1+s+\alpha$ then we reach the same conclusion by using both Theorems~\ref{thmclassif-fully} and \ref{thmclassif-fully2}.

In any case, passing \eqref{nondegvk8} to the limit and using the growth control we have
\[ \frac 12 \le \bigl\|\nabla v_k\bigr\|_{L^\infty(B_{1})} \le 1\]
and thus  $\frac{1}{2} \le (1+s)K_0 \le 1$.

Therefore, we have shown that $v_k\rightarrow K_0 (x\cdot e)_+^{1+s}$ in the $C^1$ norm, uniformly on compact sets.
In particular, \eqref{contr8} is contradicted for large $k$, and thus the lemma is proved.
\end{proof}

\begin{proof}[Proof of Proposition \ref{proprescalingsfully}]
It is almost identical to the proof of Proposition \ref{proprescalings} but we use Lemma \ref{lemconvergencefully} instead of Lemma \ref{lemconvergence}.
\end{proof}

\subsection{$C^1$ regularity of the free boundary}

We show first that the free boundary is Lipschitz in a neighborhood of any regular point $x_0$ satisfying \eqref{non-singular}.

\begin{prop}\label{free-bdry-Lipschitzfully}
Let $u$  be a solution of the obstacle problem \eqref{obstacleproblem2fully}-\eqref{ffully}-\eqref{C-1-taufully} and assume that $x_0=0$ is a regular free boundary point with exponent $\alpha$ and modulus $\nu$.
In case $1+s+\alpha\geq 2s+\bar\alpha$, assume in addition that \eqref{non-singular} holds.

Then, there exists a vector $e\in S^{n-1}$ such that for any $\ell >0$ there is $r>0$ such that
\[ \{u>0\}\cap B_r  = \bigl\{\bar x_n > g(\bar x_1, \bar x_2,\dots, \bar x_{n-1}) \bigr\}\cap B_r \]
where $\bar x  = R x$, $R$ rotation with $Re = e_n$, and where $g$ is Lipschitz with
\[\|g\|_{C^{0,1}(B_r)}\le \ell.\]
Moreover,
\[\partial_{e'}u\geq0\quad\textrm{in}\ B_r,\quad\textrm{for all}\ e'\cdot e\geq\frac{\ell}{\sqrt {1+\ell^2}},\]
and
\[\partial_e u\geq cr^s\quad\textrm{in}\ B_{r}(2re).\]
The constants $c$ and $r$ depend only on $n$, $s$, $\lambda$, $\Lambda$, $\nu$, and $\ell$.
\end{prop}

\begin{proof}
The proof is a minor modification of that of Proposition \ref{free-bdry-Lipschitz}.
\end{proof}

Using the previous result we next find the following.

\begin{prop}\label{contagifully}
Let $u$ be a solution of the obstacle problem \eqref{obstacleproblem2fully}-\eqref{ffully}-\eqref{C-1-taufully} and assume that $0$ is a regular free boundary point with exponent $\alpha$ and modulus $\nu$.
In case $1+s+\alpha\geq 2s+\bar\alpha$, assume in addition that \eqref{non-singular} holds.

Then, there is $r>0$ such that every point $x_0$ on $\partial \{u>0\}\cap B_r$ is regular and satisfies \eqref{non-singular}, with a common modulus of continuity $\tilde\nu$.
In particular, the set $\partial \{u>0\}\cap B_r$ is $C^1$, with a uniform modulus of continuity.
\end{prop}

\begin{proof}
The result follows Proposition \ref{free-bdry-Lipschitzfully} and using the homogeneous solution $\Phi$ of Lemma \ref{homog-subsol}; see the proof of Propositions \ref{contagi}.
\end{proof}

\subsection{Proof of Theorem \ref{th3}}

We now prove the $C^{1,\gamma}$ regularity of free boundaries.
For this, we will use Theorems 1.5 and 1.6 in \cite{RS-C1}.

\begin{proof}[Proof of Theorem \ref{th3}]
After subtracting the obstacle $\varphi$ and dividing by a constant, we may assume that $u$ satisfies \eqref{obstacleproblem2fully}-\eqref{ffully}-\eqref{C-1-taufully}.

Let $x_0$ be any free boundary point.
If (iii) holds then there is nothing to prove.
From now on we assume (iii) does not hold, and thus $x_0$ is a regular free boundary point.
We also assume $x_0=0$.

\vspace{3mm}

\emph{Case 1}. Assume
\[\liminf_{r\to0}\frac{|\{u=0\}\cap B_r|}{|B_r|}=0.\]
We need to show that
\begin{equation}\label{want-fully}
u(x)=o(|x|^{\min(2s+\gamma,1+s+\alpha)}).
\end{equation}
Assume that \eqref{want-fully} does not hold, and let $\alpha'$ be such that $1+s+\alpha'=\min(2s+\gamma,1+s+\alpha)$.
Notice that necessarily we have $\alpha'\geq0$, since otherwise there are no such free boundary points (by Theorems~\ref{thmclassif-fully} and \ref{thmclassif-fully2}).
Thus, $x_0=0$ is a regular point with exponent $\alpha'\geq0$, and $1+s+\alpha'<2s+\bar\alpha$.
Therefore, since $1+s+\alpha'<2s+\bar\alpha$ we do not need assumption \eqref{non-singular} in Proposition \ref{contagifully}, and hence we find that the free boundary will be $C^1$ near $0$.
But then
\[\lim_{r\to0}\frac{|\{u=0\}\cap B_r|}{|B_r|}=\frac12,\]
a contradiction.
Hence, \eqref{want-fully} is proved.

\vspace{3mm}

\emph{Case 2}. Assume now
\begin{equation}\label{non-singular-proof}
\liminf_{r\to0}\frac{|\{u=0\}\cap B_r|}{|B_r|}>0.
\end{equation}
Then, by Proposition \ref{contagifully}, the set of regular points satisfying \eqref{non-singular-proof} is relatively open and the free boundary is $C^1$ near those points.

Furthermore, rescaling exactly as in the proof of Theorem \ref{th2} and using Theorem 1.6 in \cite{RS-C1}, we find that the free boundary is $C^{1,\gamma}$ for all $\gamma\in(0,\bar\alpha)$ in a neighborhood of $0$.
Finally, thanks to Theorem 1.5 in \cite{RS-C1} we have
\[\partial_i u/d^s\in C^\gamma(\{u>0\}\cap B_r)\]
for some $r>0$, and this yields
\[u(x)=c\,d^{1+s}+o(|x|^{1+s+\gamma}),\]
as desired.
\end{proof}

\section{Appendix: Proof of Proposition \ref{liouville}}

We need to introduce the following definition.
Given
$\Omega\subset \R^n$ open, let
\[ [u]_{W^{\sigma, \infty}(\Omega)} =  {\rm ess\,sup }_{x\in\Omega} \int_{B_{d_x}} \bigl| \partial^2 u(x,y) \bigr|  (2-\sigma)|y|^{-n-\sigma} \,dy \]
where $d_x={\rm dist}(x,\Omega)$ and
\[ \partial^2 u(x,y) = u(x+y) + u(x-y) -2u(x) \]
Note that if $v\in W^{\sigma, \infty}(\Omega)$ and $|v(x)|\le 1+|x|^{\sigma-\epsilon}$ then
\[ M^{+}_{\mathcal L_0}v(x)  : = \Lambda \int_{\R^n} \bigl( \partial^2 v(x,y)\bigr)^+\,dy -\lambda \int_{\R^n} \bigl( \partial^2 v(x,y)\bigr)^- \,dy\]
and
\[ M^{-}_{\mathcal L_0}v(x)  : = \lambda \int_{\R^n} \bigl( \partial^2 v(x,y)\bigr)^+\,dy -\Lambda \int_{\R^n} \bigl( \partial^2 v(x,y)\bigr)^- \,dy\]
are defined for almost every $x\in \Omega$.

\begin{proof}[Proof of Proposition \ref{liouville}]
The result for all $C_1>0$ trivially follows from the result for $C_1=1$. Thus, in all the proof we assume that $C_1=1$.
Throughout the proof we will assume without loss of generality that $u$ is $C^2$ at $0$. Indeed, since $u$ is convex by a classical theorem of Alexandrov-Bussemann-Feller $u$ is second order differentiable at a.e. point. Thus, if the origin is not a good point we set the origin at a new point in $B_1$ (changing  $C_1$ by $2^{\sigma+\gamma}C_1$).
More precisely, after subtracting a plane we may assume that
\[  0 =u(0) \le u(x) \le M|x|^2 \quad \mbox{in }B_1\]
for some $M$ large enough. We will not need use any quantitative control on $M$ but we only need in the proof that $M<\infty$ so that certain viscosity solutions are satisfy the equation in the integral sense at almost every point.

{\em Step 1.}
For fixed $h \in \R^n$,  taking $\mu$ a mass concentrated at $\pm h$ we obtain that, in the viscosity sense,
\[ M^{+}_{\mathcal L_0} (u - u(\cdot-h)) \ge  0 \quad \mbox{in }\R^n\]
and
\[ M^{+}_{\mathcal L_0} (u - u(\cdot+h)) \ge  0 \quad \mbox{in }\R^n.\]
Thus
\[ M^{-}_{\mathcal L_0} (u - u(\cdot-h)) = - M^{+}_{\mathcal L_0} (u(\cdot-h)-u) \le  0 \quad \mbox{in }\R^n.\]
Therefore, for $v = u - u(\cdot-h)$ we have
\[ M^- _{\mathcal L_0} v \le 0 \le M^+_{\mathcal L_0} v\]
in the viscosity sense.

{\em Step 2. }
We first show that $u\in W^{\sigma, \infty}(B_R)$ for all $R\ge 1$ and prove bounds for the corresponding seminorms.

Given $\rho\ge 1$ we consider the rescaled function
\[\bar u(x) = \rho^{-\sigma-\gamma} u(\rho x)\]
It is immediate to verify that $\bar u$ satisfies the same assumptions (i), (ii), and (iii) as $u$. In particular the constant $C_1$ in (i) for $\bar u$ is the same as that of $u$, that is $C_1=1$.

Since $u \le 1$ in $B_1$, the parabola $|2x|^2 +c$ touches $\bar u$ by above in $B_1$ (for some $c>0$) at the point $x_0\in\overline{B_{1/2}}$.
Given $v = \bar u - \bar u(\cdot-h)$, since $u$ is convex and can be touched by a parabola by above at $x_0$, the function $v$ can be touched by a parabola by below at $x_0$.
Then, by Lemma 3.3 in \cite{CS} the Pucci operator can be evaluated at
this point and we have
\[
M^+v (x_0) = \int_{\R^n}   \delta^2 v(x_0,y)   \frac{b(y)}{|y|^{n+\sigma}}\, dy \ge 0
\]
where
\[b(y):=
\begin{cases}
 \Lambda  \quad  &\mbox{if } \delta^2 v(x_0,y) >0 \\
 \lambda 		&\mbox{if}  \delta^2 v(x_0,y) <0.
\end{cases}
\]

Now we rewrite this as
\[
\int_{B_1}   \delta^2 \bar u(x_0,y)   \frac{b(y)}{|y|^{n+\sigma}}\, dy  + \int_{\R^n\setminus B_1}  \delta^2 v(x_0,y)  \frac{b(y)}{|y|^{n+\sigma}}\, dy\ge \int_{B_1}   \delta^2 \bar u(x_0+h,y)   \frac{b(y)}{|y|^{n+\sigma}}\, dy.
\]
Using the convexity of $\bar u$, the fact that $|2x|^2 +c$ touches $\bar u$ by above at $x_0$ in $B_{1/2}(x_0)\subset B_1$, and recalling that
\[
\begin{split}
|\delta^2 v(x_0,y)| &\le |u(x_0+y)-u(x_0+h+y) | + |u(x_0-y)-u(x_0+h-y) | \,+
\\
& \hspace{75 mm}+ 2 |u(x_0)-u(x_0+h)| \\
&\le 4(1+|y|^{\sigma+\gamma-1})
\end{split}
\]
 for $|h|\le 1$ by (i) we obtain
\[
\int_{B_{1/2}}   8|y|^2   \frac{\Lambda}{|y|^{-n-\sigma}}\, dy  + \int_{\R^n\setminus B_1}  4(1+|y|^{\sigma+\gamma-1})  \frac{b(y)}{|y|^{-n-\sigma}}\, dy\ge \int_{B_{1/2}}   \delta^2 \bar u(x_0+h,y)   \frac{\lambda}{|y|^{n+\sigma}}\, dy
\]
for all $h\in B_1$.

Thus,
\[
 \int_{B_{1/2}}   \delta^2 \bar u(x,y)   \frac{1}{|y|^{n+\sigma}}\, dy \le C
\]
for all $x\in B_{1/2}$  with $C$ universal (this meaning that it depends only on $n$, $\sigma_0$, $\lambda$, and $\Lambda$).
This implies that
\[
[\bar u]_{W^{\sigma,\infty}(B_{1/2})} \le C.
\]

This implies, rescaling from $\bar u$ to $u$ and taking $\rho = 2R$
\[
[u]_{W^{\sigma,\infty}(B_{R})} \le CR^{\gamma} .
\]

{\em Step 3.}
For $t\ge0$, let us define
\[P^t(x):=  \int_{\R^n}\bigl(\delta^2u(x,y)-\delta^2u(0,y)\bigr)^+\, \frac{2-\sigma}{(t+ |y|)^{n+\sigma}}\,dy \]
and
\[N^t(x):=  \int_{\R^n}\bigl(\delta^2u(x,y)-\delta^2u(0,y)\bigr)^- \,\frac{2-\sigma}{(t+ |y|)^{n+\sigma}}\,dy .\]

By Step 2  we have
\begin{equation}\label{growthPN0}
  0\le  P^0 \le CR^{\gamma}\quad \mbox{and}\quad   0\le  N^0 \le CR^{\gamma} \quad \mbox{in }B_R,
\end{equation}
for all $R\ge 1$, with $C$ universal (depending only on $n$, $\sigma_0$, $\lambda$, and $\Lambda$).

Next, dividing $u$ by the universal constant $C$ in \eqref{growthPN0} we may assume
\begin{equation} \label{growthPN}
  0\le  P^0 \le 4^{k\gamma} \le 2^{k\bar \alpha} \quad \mbox{in } B_{4^k}(0) \quad \mbox{for all }k\ge0.
\end{equation}
and the same for $N^0$.
We will prove that, if  $\bar \alpha$ is taken small enough then
\begin{equation} \label{dyadicimprovement}
  0\le  P^0 \le 4^{k\bar \alpha} \quad \mbox{in } B_{4^k}(0) \quad \mbox{for all }k\le-1.
\end{equation}
and the same for $N^0$.
This and a scaling argument will easily lead to a contradiction with the growth \eqref{growthPN0} since $\gamma<\bar\alpha$ unless $P^0\equiv N^0 \equiv 0$ in all of $\R^n$.

This estimate \eqref{dyadicimprovement} on $P^0$ is proved though an iterative improvement on the maximum of $P^0$ on dyadic balls.

Indeed, our goal is to improve the  bound from above $P^0\le1$ in  $B_1$ to  $P^0\le 1-\theta$ in $B_{1/4}$, for some $\theta>0$.
After doing this, we will immediately have \eqref{dyadicimprovement} for all $k\ge1$ for some $\bar\alpha$ small (related  to $\theta$) just by scaling and iterating.
Let us thus concentrate in proving  $P^0\le 1-\theta$ in $B_{1/4}$.

Note that $P^t \le P^0$ for all $t>0$ and that $P^0= \lim_{t\to 0} P^t$ by monotone convergence.

We will assume that $P^0(x_0) \ge 1-2\theta$ for some $x_0\in B_{1/2}$. We will reach a contradiction taking $\theta$ small enough.

Define the set
\[ A = \{ y : (u(x_0+y)+u(x_0-y)-2u(x_0) - u(y) - u(-y) + 2u(0)) > 0 \}.\]
 In particular we have
\begin{align*}
P^t(x_0) = \int_{A} \bigl(\delta^2u(x_0,y)-\delta^2u(0,y)\bigr) \frac{2-\sigma}{(t+ |y|)^{n+\sigma}} dy, \\
N^t(x_0) = \int_{\R^n \setminus A} \bigl(\delta^2u(x_0,y)-\delta^2u(0,y)\bigr) \frac{2-\sigma}{(t +|y|)^{n+\sigma}} dy.
\end{align*}

We will take  $\bar \alpha$ very small (depending on $\delta_0$ below) so that \eqref{growthPN} implies
\begin{equation} \label{eeeerror}
\int_{\R^n} \bigl(P^t(y)-1\bigr)^+ \frac{2-\sigma}{|y|^{n+\sigma}} \,dy   \le \delta_0.
\end{equation}

We define $v^t$ as
\[ v^t(x) := \int_{A} \bigl(\delta^2u(x,y)-\delta^2u(0,y)\bigr) \frac{2-\sigma}{(t+|y|)^{n+\sigma}} dy. \]
Note that in particular  $P^t(x_0)=v^t(x_0)$.
Let
\begin{equation}\label{bartheta}
\bar \theta = \frac{\lambda}{4\Lambda}
\end{equation}
and define the set
\[\boldsymbol D := \{ x\in B_{1/2}\, :\, v^0 \ge (1-\bar \theta)\}.\]

Let us show that, given  $\eta>0$ we can take $\theta>0$ small enough so that we have
\begin{equation}\label{Asmall}
|\boldsymbol D|\ge (1-\eta) |B_1|.
\end{equation}

Let  $t>0$ small and $x_0 \in B_{1/4}$ be such that $P^t(x_0) = 1-3\theta$.

Note now that \eqref{Asmall} is equivalent to
\begin{equation}\label{equiv}
 \bigl| \{ x\in B_{1/2}\, :\, v^0 \le (1-\bar \theta)\}\bigr| \le \eta |B_1|.
\end{equation}
Let us prove this.

By (iii), approximating $\chi_{A}(y) (t+|y|)^{-n-\sigma}$ by $L^1$ functions $\mu$ with compact support and using the stability under uniform convergence result for subsolutions  \cite[Lemma 4.3]{CS2} we show that
\[ M^+_{\mathcal L_0} v^t \geq 0  \quad \mbox{in }\R^n. \]

We now apply the nonlocal $L^\varepsilon$ Lemma of Theorem 10.4 in \cite{CS} to the function $(1-v^t)_+$, which is an approximate supersolution in $B_{3/4}$ ---with right hand side $C\delta_0$   in \eqref{eeeerror}---. Note that this function is nonnegative in all of $\R^n$. We obtain
\[   \int_{B_{1/2}}  (1-v^t)_+^\varepsilon \le C( \theta +\delta_0)^{\varepsilon}. \]
Then, by Fatou's Lemma,
\[   \int_{B_{1/2}}  (1-v^0)_+^\varepsilon \le C( \theta +\delta_0)^{\varepsilon}. \]

Taking now $\theta$ and $\delta_0$ small enough we obtain \eqref{equiv}.

We now will obtain a contradiction from \eqref{Asmall} where $\eta$ is small to be chosen later.  We have that $v$ is larger than $(1-\bar\theta)$ in most of  $B_{1/2}$.
In that case we consider the function $w^t$ defined as $v^t$ but replacing $A$ by $\R^n\setminus A$.
\[ w^t(x) := \int_{\R^n \setminus A} \bigl(\delta^2u(x,y)-\delta^2u(0,y)\bigr)  \frac{2-\sigma}{|y|^{n+\sigma}}dy. \]

Using (iii), approximating $\chi_{\R^n\setminus A}(y) (2-\sigma) |y|^{-n-\sigma}$ by $L^1$ functions $\mu$ with compact support and using the stability under uniform convergence result for subsolutions  \cite[Lemma 4.3]{CS2} we show that
\[ M^+_{\mathcal L_0} w^t \geq 0  \quad \mbox{in }\R^n\]
for all $t>0$.

We observe that  by definition $P^0-N^0 = v^0+w^0$ and that, we have \[0 \leq P - v \leq 1-(1-\bar\theta)\le \bar \theta\quad \mbox{in }\boldsymbol D\] ---here we have used that $P^0\le 1$ in $B_1$ by \eqref{growthPN} .

In addition, reasoning similarly as in Step 2, and recalling that $u$ is second oder differentiable at 0  the viscosity inequalities
\[ M^{-}_{\mathcal L_0} (u - u(\cdot-h)) \le 0 \le M^{+}_{\mathcal L_0} (u - u(\cdot-h))  \]
for arbitrary $h$ imply the pointwise integral inequalities
\begin{equation}\label{PNcomparable}
\frac{\lambda}{\Lambda} P^0(x) \le N^0(x) \le \frac{\Lambda}{\lambda} P^0(x)
\end{equation}
for almost every $x$.
Here we are using again the convexity of $u$ and applying the Alexandrov-Bussemann-Feller Theorem and the Lemma 3.3 from \cite{CS}.

Therefore,
\[\begin{split}
 w &= (P-v) -N \le \bar \theta -N \le \bar\theta -\frac{\lambda}{\Lambda}P  \\
&\le \bar \theta - \frac{\lambda}{\Lambda} (1-\bar \theta)\\
&\le - \lambda/\Lambda + 2\bar \theta \le -c  \quad \mbox{ a.e. in } \boldsymbol D,
\end{split}\]
where $c= \lambda/2\Lambda >0$. Here we have used \eqref{bartheta}.

We thus may take $t$ small enough so that
\[ | \{ w^t > -3c/4\} |\le 2\eta    \]

We now use the ``half'' Harnack of Theorem 5.1 in \cite{CS3} applied to the function $\bar w = \bigl(w^t(r\,\cdot\,)+3c/4\bigr)^+$ (with $r>0$ small) to conclude that $w^t(0)+3c/4 \le c/2$.
Indeed, the function $\bar w$ is a subsolution and, by \eqref{growthPN},  it satisfies $0\le \bar w \le  P +3c/4  \le 2^{k\bar \alpha} \quad \mbox{in } B_{2^k/r}(0)$ and $\bar w =0$ in $\boldsymbol D/r$, which covers most of $B_{1/r}$.
Hence, taking both $r$ and $\eta$ small enough we can make $\int_{\R^n} \bar w^t(y) \omega_\sigma(y) \,dy$ as small as we wish. Thus, using Theorem 5.1  in \cite{CS3} we find that $w^t(0)+3c/4 \le c/2$ as promised.
As a consequence we obtain that $w^t(0)\le -c/4<0$; a contradiction since $w^t(0)=0$ by definition.

Therefore \eqref{dyadicimprovement} follows.

{\em Step 4.} Applying the previous Steps to the rescaled functions $\bar u =\rho^{-\sigma-\gamma} u(\rho\,\cdot\,)$ we find that
\begin{equation} \label{dyadicimprovement2}
  0\le  \bar P^0 \le C 4^{k\bar \alpha} \quad \mbox{in } B_{4^k}(0) \quad \mbox{for all }k\le-1.
\end{equation}
and the same for $\bar N^0$, where
\[\bar P^0(x):=  \int_{\R^n}\bigl(\delta^2\bar u(x,y)-\delta^2\bar u(0,y)\bigr)^+\, \frac{2-\sigma}{ |y|^{n+\sigma}}\,dy \]
and
\[\bar N^0(x):=  \int_{\R^n}\bigl(\delta^2\bar u(x,y)-\delta^2\bar u(0,y)\bigr)^- \,\frac{2-\sigma}{ |y|^{n+\sigma}}\,dy .\]

This implies that
\[ \sup_{B_{4^{-l} \rho}} P_0 =  \rho^{\gamma} \sup_{B_{4^{-l}}} \bar P_0  \le C \rho^\gamma 4^{-l\bar\alpha },\quad l\ge 1.\]
Thus, taking $\rho= R_0 4^{l}$ we obtain
\[ \sup_{B_{R_0}} P_0 \le  \limsup_{l\to +\infty} C (R_0)^\gamma 4^{-l(\bar\alpha-\gamma)} =0.\]
Since this is for arbitrary $R_0$ we obtain $P_0\equiv 0$. Similarly $N_0\equiv 0$. This implies that $\delta^2 u(x,y)$ is constant in $x$ and thus $u$ is a quadratic polynomial.
\end{proof}


\begin{thebibliography}{00}

\bibitem[BFR15]{BFR} B. Barrios, A. Figalli, X. Ros-Oton, \emph{Global regularity for the free boundary in the obstacle problem for the fractional Laplacian}, preprint arXiv (June 2015).

\bibitem[CSS08]{CSS} L. Caffarelli, S. Salsa, L. Silvestre, \emph{Regularity estimates for the solution and the free boundary of the obstacle problem for the fractional Laplacian}, Invent. Math. 171 (2008), 425-461.

\bibitem[CS09]{CS} L. Caffarelli, L. Silvestre, \emph{Regularity theory for fully nonlinear integro-differential equations}, Comm. Pure Appl. Math. 62 (2009), 597-638.

\bibitem[CS11]{CS3} L. Caffarelli, L. Silvestre, \emph{The Evans-Krylov theorem for nonlocal fully nonlinear equations}, Ann. of Math. 174 (2011), 1163-1187.

\bibitem[CS11b]{CS2} L. Cafarelli, L. Silvestre, \emph{Regularity results for nonlocal equations by approximation}, Arch. Rat. Mech. Anal. 200 (2011), 59-88.

\bibitem[CT04]{CT} R. Cont, P. Tankov, \emph{Financial Modelling With Jump Processes}, Chapman \& Hall, Boca Raton, FL, 2004.

\bibitem[GP09]{GP} N. Garofalo, A. Petrosyan, \emph{Some new monotonicity formulas and the singular set in the lower dimensional obstacle problem}, Invent. Math. 177 (2009), 415-461.

\bibitem[RS14]{RS-K} X. Ros-Oton, J. Serra, \emph{Boundary regularity for fully nonlinear integro-differential equations}, Duke Math. J. {165} (2016), 2079-2154.

\bibitem[RS14b]{RS-stable} X. Ros-Oton, J. Serra, \emph{Regularity theory for general stable operators}, J. Differential Equations {260} (2016), 8675-8715.

\bibitem[RS15]{RS-C1} X. Ros-Oton, J. Serra, \emph{Boundary regularity estimates for nonlocal elliptic equations in $C^1$ and $C^{1,\alpha}$ domains}, preprint arXiv (Dec. 2015).

\bibitem[Ser14]{S-convex} J. Serra, \emph{$C^{\sigma+\alpha}$ regularity for concave nonlocal fully nonlinear elliptic equations with rough kernels}, Calc. Var. Partial Differential Equations {54} (2015), 3571-3601.

\bibitem[Sil07]{S-obst} L. Silvestre, \emph{Regularity of the obstacle problem for a fractional power of the Laplace operator}, Comm. Pure Appl. Math. 60 (2007), 67-112.

\end{thebibliography}
\end{document}